\PassOptionsToPackage{dvipsnames}{xcolor}
\documentclass[USenglish, a4paper, oneside, 10pt]{article}
\makeatletter
		\usepackage{amsmath}
		\usepackage{amssymb}
		\usepackage{amsthm}
		\usepackage[USenglish]{babel}
		\usepackage{enumitem}
		\usepackage{eqparbox}
		\usepackage{float}
		\usepackage[T1]{fontenc}
		\usepackage[bottom = 1.2in, left = 1.1in, right = 1.1in]{geometry}
		\usepackage{graphicx}
		\usepackage{mathtools}
		\usepackage{nicefrac}
		\usepackage{pgfplots}
		\usepackage{xcolor}
		\usepackage{hyperref}
		\usepackage{algorithm}
		\usepackage[capitalize,nameinlink]{cleveref}
		\usepackage[noend]{algpseudocode}

		\newcommand{\TheTitle}{An Adaptive Linesearch-free Method for Monotone Variational Inequalities under Local Lipschitz Continuity}

		\title{%
			\TheTitle\thanks{%
				A. Themelis was supported by the JSPS KAKENHI grant number JP24K20737.
				P. Latafat is a member of the Gruppo Nazionale per l'Analisi Matematica, la Probabilit\`a e le loro Applicazioni (GNAMPA - National Group for Mathematical Analysis, Probability and their Applications) of the Istituto Nazionale di Alta Matematica (INdAM - National Institute of Higher Mathematics).%
			}%
		}

		\date{}

		\author{%
			Hongjia Ou\thanks{%
				Faculty of Information Science and Electrical Engineering (ISEE),
				Kyushu University, 744 Motooka, Nishi-ku 819-0395, Fukuoka, Japan.
				\email{ou.honjia.069@s.kyushu-u.ac.jp},
				\email{andreas.themelis@ees.kyushu-u.ac.jp}%
			}%
		\and
			Andreas Themelis\footnotemark[2]
		\and
			Puya Latafat\thanks{%
				IMT School for Advanced Studies Lucca, Piazza S. Francesco 19, 55100 Lucca, Italy.
				\email{puya.latafat@imtlucca.it}%
			}%
		}

		\ifpdf
			\hypersetup{
				pdftitle = {\TheTitle},
				pdfauthor = {H. Ou, A. Themelis, and P. Latafat},
			}
		\fi

		\graphicspath{{Pics/Tikz/}}

		\hypersetup{%
			breaklinks,
			hypertexnames = true,
			plainpages    = false,
			colorlinks = true,
			urlcolor   = MidnightBlue,
			linkcolor  = MidnightBlue,
			citecolor  = ForestGreen,
		}

		\SetEnumitemKey{widestL}{
			labelwidth=\eqboxwidth{listlabel@\EnumitemId},
			leftmargin=!,
			format=\mylistlabel,
		}
		\newcommand\mylistlabel[2][l]{\eqmakebox[listlabel@\EnumitemId][#1]{#2}}

		\setlist[itemize,1]{
			wide,
			label=\raisebox{0.3ex}{$\bullet$},
			leftmargin = 0pt,
			labelindent = 0pt,
		}
		\setlist[itemize,2]{
			wide,
			label=\raisebox{0.3ex}{\tiny$\bullet$},
			leftmargin = 0pt,
			labelindent = 0pt,
		}

		\setlist[enumerate,1]{
			label=\text{\upshape{(\roman*)}},
		}
		\setlist{
			widestL,
			topsep=1ex,
			itemsep=0pt,
		}

		\newlist{enumerateq}{enumerate}{1}
		\setlist[enumerateq]{
			label={(\alph*)},
		}

		\newlist{enumeratass}{enumerate}{1}
		\setlist[enumeratass]{
			label={{\sc a}\oldstylenums{\arabic*}},
		}
		\Crefname{enumeratassi}{Assumption}{Assumptions}

		\newcommand{\numberthis}{\addtocounter{equation}{1}\tag{\theequation}}

		\newcommand{\createthmlists}[1]{%
			\newlist{#1enumerate}{enumerate}{1}%
			\setlist[#1enumerate,1]{%
				label = \upshape{(\roman*)},%
				ref = \text{\thetheorem(\roman*)},%
			}%
			\crefalias{#1enumeratei}{#1}%
			\newlist{#1enumerateq}{enumerate}{1}%
			\setlist[#1enumerateq,1]{%
				label = \upshape{(\alph*)},%
				ref = \text{\thetheorem(\alph*)},%
			}%
			\crefalias{#1enumerateqi}{#1}%
			\AtBeginEnvironment{#1}{%
				\expandafter\let\expandafter\enumerate\csname #1enumerate\endcsname
				\expandafter\let\expandafter\endenumerate\csname end#1enumerate\endcsname
				\expandafter\let\expandafter\enumerateq\csname #1enumerateq\endcsname
				\expandafter\let\expandafter\endenumerateq\csname end#1enumerateq\endcsname
			}{}{}%
		}%

		\createthmlists{corollary}
		\createthmlists{example}
		\createthmlists{fact}
		\createthmlists{lemma}
		\createthmlists{remark}
		\createthmlists{theorem}

		\let\oldstate\State
		\renewcommand{\State}{\normalfont\normalcolor\normalsize\oldstate}
		
		\renewcommand{\theHALG@line}{\thealgorithm.\arabic{ALG@line}}
		
		\def\fixalgref{%
			\phantomsection
			\addtocounter{ALG@line}{-1}%
			\refstepcounter{ALG@line}%
		}
		\let\clevethm@OLDState\State  \renewcommand{\State}{\clevethm@OLDState\fixalgref}
		\crefname{ALG@line}{step}{steps}
		\Crefname{ALG@line}{Step}{Steps}

		\newcommand{\algnamefont}{\sffamily}
		\newcommand{\EG}{{\algnamefont EG}}
		
		\newcommand{\FBF}{{\algnamefont FBF}}
		\newcommand{\frb}{{\algnamefont FRB}}
		
		\newcommand{\AGRAAL}{{\algnamefont aGRAAL}}

		\NewDocumentCommand{\FRB}{E{_}{\alpha}}{{%
			\algnamefont
			FRB$_{#1\vphantom{2}}$%
		}}%

		\NewDocumentCommand{\adaFRB}{t{+} E{_}{\alpha}}{{%
			\algnamefont
			adaFRB$_{#2\vphantom{2}}^{\IfBooleanTF{#1}{+}{\phantom{+}}}$%
		}}%
		\NewDocumentCommand{\AdaFRB}{t{+} E{_}{\alpha}}{{%
			\algnamefont
			AdaFRB$_{#2\vphantom{2}}^{\IfBooleanTF{#1}{+}{\phantom{+}}}$%
		}}
		\NewDocumentCommand{\refadaFRB}{t{+} E{_}{\alpha}}{{%
			\IfBooleanTF{#1}{%
				\hyperref[alg:adaFRB+]{\algnamefont adaFRB$_{#2\vphantom{2}}^+$}%
			}{%
				\hyperref[alg:adaFRB]{\algnamefont adaFRB$_{#2\vphantom{2}}^{\vphantom{+}}$}%
			}%
		}}
		\newcommand{\refadaFRBII}{{%
			\hyperref[alg:adaFRB+2]{\algnamefont adaFRB$_{2}^+$}%
		}}
		\newcommand{\refAdaFRBII}{{%
			\hyperref[alg:adaFRB+2]{\algnamefont AdaFRB$_{2}^+$}%
		}}

		\NewDocumentCommand{\refAdaFRB}{t{+} E{_}{\alpha}}{{%
			\IfBooleanTF{#1}{%
				\hyperref[alg:adaFRB+]{\algnamefont AdaFRB$_{#2\vphantom{2}}^+$}%
			}{%
				\hyperref[alg:adaFRB]{\algnamefont AdaFRB$_{#2\vphantom{2}}^{\vphantom{+}}$}%
			}%
		}}

			\DeclareSymbolFont{varmathbb}{U}{txmia}{m}{it}
			\DeclareMathSymbol{\m@thbbch@rC}{\mathord}{varmathbb}{131}
			\DeclareMathSymbol{\m@thbbch@rN}{\mathord}{varmathbb}{142}
			\DeclareMathSymbol{\m@thbbch@rR}{\mathord}{varmathbb}{146}
		\newcommand{\C}{\m@thbbch@rC}
		\newcommand{\N}{\m@thbbch@rN}
		\newcommand{\R}{\m@thbbch@rR}
		\newcommand{\Rinf}{\overline{\R}}
		\newcommand{\V}{\mathcal{V}}

		\newcommand{\gammin}{\gamma_{\rm min}}
		\newcommand{\rhomax}{\rho_{\rm max}}

		\DeclareMathOperator*{\argmin}{argmin}
		\DeclareMathOperator{\diag}{diag}
		\DeclareMathOperator{\dist}{dist}
		\DeclareMathOperator{\dom}{dom}
		\DeclareMathOperator{\Gap}{Gap}
		\DeclareMathOperator{\prox}{prox}
		\DeclareMathOperator{\zer}{zer}

		\newcommand{\ball}[2]{\operatorname{\overline B}(#1,#2)}

		\def\Dg(#1,#2){\operatorname{\tilde D}_g(#1,#2)}

		\newcommand{\func}[3]{#1:#2\rightarrow#3}

		\let\oldnabla\nabla
		\renewcommand{\nabla}{\@ifstar{\widetilde{\oldnabla}}{\oldnabla}}

		\NewDocumentCommand{\innprod}{s m m}{%
			\IfBooleanT{#1}{\left}\langle
				#2,#3
			\IfBooleanT{#1}{\right}\rangle
		}

		\NewDocumentCommand{\seq}{s m e{_^} o}{%
			\IfBooleanT{#1}{\left}(
				#2
			\IfBooleanT{#1}{\right})
			\IfValueTF{#3}{_{#3}}{_{k\in\N}}
			\IfValueT{#4}{^{#4}}
			\IfValueT{#5}{{\color{red}!!DEPRECATED!!}} 
		}

		\NewDocumentCommand{\set}{s m o}{%
			\IfBooleanT{#1}{\left}\{
				#2
				\IfValueT{#3}{\mid#3}
			\IfBooleanT{#1}{\right}\}
		}

		\NewDocumentCommand{\newkvar}{s m m t{,}}{%
			\IfBooleanTF{#1}{%
				\NewDocumentCommand{#2}{s t' t' e{^_}}{%
					\IfBooleanTF{##1}{%
						\def\newkvar@sup{k+1}%
					}{%
						\IfBooleanTF{##2}{%
							\IfBooleanTF{##3}{%
								\def\newkvar@sup{k-2}%
							}{%
								\def\newkvar@sup{k-1}%
							}%
						}{%
							\IfValueTF{##4}{%
								\IfBooleanTF{#4}{%
									\def\newkvar@sup{k}%
								}{%
									\def\newkvar@sup{##4}%
								}%
							}{%
								\def\newkvar@sup{k}%
							}%
						}%
					}%
					\def\newkvar@sup@{\newkvar@sup}%
					\IfBooleanT{#4}{\IfValueT{##4}{\def\newkvar@sup@{\newkvar@sup,##4}}}%
					#3^{\newkvar@sup@}%
					\IfValueT{##5}{_{##5}}%
				}%
			}{%
				\NewDocumentCommand{#2}{s t' t' e{_^}}{%
					\IfBooleanTF{##1}{%
						\def\newkvar@sub{k+1}%
					}{%
						\IfBooleanTF{##2}{%
							\IfBooleanTF{##3}{%
								\def\newkvar@sub{k-2}%
							}{%
								\def\newkvar@sub{k-1}%
							}%
						}{%
							\IfValueTF{##4}{%
								\IfBooleanTF{#4}{%
									\def\newkvar@sub{k}%
								}{%
									\def\newkvar@sub{##4}%
								}%
							}{%
								\def\newkvar@sub{k}%
							}%
						}%
					}%
					\def\newkvar@sub@{\newkvar@sub}%
					\IfBooleanT{#4}{\IfValueT{##4}{\def\newkvar@sub@{\newkvar@sub,##4}}}%
					#3_{\newkvar@sub@}%
					\IfValueT{##5}{^{##5}}%
				}%
			}%
		}

		\newkvar{\betk}{\beta}
		\newkvar{\ck}{c}
		\newkvar{\delk}{\delta}
		\newkvar{\epsk}{\varepsilon}
		\newkvar*{\DFk}{\Delta_F}
		\newkvar*{\Dxk}{\Delta_x}
		\newkvar{\gamk}{\gamma}
		\newkvar{\lk}{\ell}
		\newkvar{\Lk}{L}
		\newkvar{\rhok}{\rho}
		\newkvar{\hrhok}{\widehat\rho}
		\newkvar{\sigk}{\sigma},
		\newkvar{\tauk}{\tau},
		\newkvar{\thetk}{\vartheta}
		\newkvar*{\uk}{u}
		\newkvar*{\wk}{w}
		\newkvar*{\xk}{x}
		\newkvar{\Uk}{\mathcal U^\alpha}
		\newkvar{\Vk}{V}

		\let\oldvert\|
		\renewcommand{\|}{{\red\Bigg\oldvert}}
		\DeclarePairedDelimiter\norm{\lVert}{\rVert}
		\DeclarePairedDelimiter\abs{\lvert}{\rvert}

		\let\OLDoverbracket\overbracket
		\let\OLDunderbracket\underbracket
		\let\OLDoverbrace\overbrace
		\let\OLDunderbrace\underbrace

		\newcommand{\@dosomething}{}

		\def\overbracket {\let\@bracketcmd\OLDoverbracket \@bracketcommand}
		\def\underbracket{\let\@bracketcmd\OLDunderbracket\@bracketcommand}
		\def\overbrace   {\let\@bracketcmd\OLDoverbrace   \@bracketcommand}
		\def\underbrace  {\let\@bracketcmd\OLDunderbrace  \@bracketcommand}

		\NewDocumentCommand{\@bracketcommand}{s o d<> m e{_^}}{{%
			\IfBooleanT{#1}{%
				\let\@dosomething\smash
				\vphantom{#4}%
			}%
			\IfValueTF{#2}{%
				\def\@bracketcmd@{\@bracketcmd[#2]}%
			}{%
				\def\@bracketcmd@{\@bracketcmd}%
			}
			\IfValueTF{#3}{%
				\alt<#3>{%
					\@dosomething{%
						\@bracketcmd@{#4}%
						\IfValueT{#5}{_{#5}}%
						\IfValueT{#6}{^{#6}}%
					}%
				}{#4}%
			}{%
				\@dosomething{%
					\@bracketcmd@{#4}%
					\IfValueT{#5}{_{#5}}%
					\IfValueT{#6}{^{#6}}%
				}%
			}%
		}}

		\let\includetikz\includegraphics

		\pgfplotsset{
			compat = 1.18,
			myaxis/.style = {
				grid = both,
				minor grid style = {
					line width = 0.4pt,
					color = gray!20,
				},
				major grid style = {
					line width = 0.4pt,
					color = gray!20,
				},
				title style = {font = \small},
				label style = {font = \footnotesize},
				tick label style = {font = \scriptsize},
				every x tick scale label/.append style = {
					yshift = 0.75em,
					font = \tiny,
				},
				legend style = {
					at = {(1, 1.25)},
					anchor = south,
					draw = black,
					rounded corners = 2pt,
					fill = white,
					fill opacity = 0.5,
					text opacity = 1,
					font = \scriptsize,
					column sep = 0.6em,
				},
				legend cell align = {left},
				table/x index=0,
				table/y index=1,
			},
		}

		\newcommand{\appendixnotitle}{%
			\refstepcounter{section}%
			\section*{\appendixname~\thesection}%
			\addcontentsline{toc}{section}{\appendixname~\thesection}%
		}

		\newenvironment{keywords}{\par\medskip\textbf{Keywords}\quad}{\par}

		\newenvironment{MSCcodes}{\par\medskip{\bf MSC codes: }}{\par}

		\theoremstyle{plain}
		\newtheorem{theorem}{Theorem}[section]

			\Crefname{corollary}{Corollary}{Corollaries}
		
		\newtheorem{lemma}[theorem]{Lemma}
		\newtheorem{fact}[theorem]{Fact}
			\Crefname{fact}{Fact}{Facts}

		\theoremstyle{definition}
		\newtheorem{example}[theorem]{Example}
		\newtheorem{remark}[theorem]{Remark}

		\numberwithin{equation}{section}
		\numberwithin{algorithm}{section}
		\newcommand{\email}[1]{{\sf\href{mailto:#1}{#1}}}

		\newif\ifnoproof\noprooftrue
		\AfterEndEnvironment{proof}{\noprooffalse}
		\newcommand{\addqedto}[2][]{%
			\AtBeginEnvironment{#2}{%
				\global\noprooftrue
				\ifstrempty{#1}{}{%
					\renewcommand{\qedsymbol}{#1}%
				}%
				\pushQED{\qed}%
			}%
			\AtEndEnvironment{#2}{%
				\ifnoproof\popQED\fi
				\global\noprooftrue
			}%
		}

		\addqedto{example}
		\addqedto{remark}

\makeatother

\begin{document}

	\maketitle

	\begin{abstract}
		The forward-reflected-backward (FRB) splitting solves monotone inclusions involving the sum of a maximally monotone operator and a monotone Lipschitz continuous operator.
		Each iteration performs one resolvent step and one evaluation of the Lipschitz operator, together with a reflection term with coefficient equal to one that reuses the operator value from the previous iteration.
		We consider the setting where the maximally monotone operator is the subdifferential of a proper, closed, convex function.
		We identify the admissible range of constant reflection coefficients, namely all values larger than one half for a sufficiently small stepsize.
		We then propose a linesearch-free adaptive variant of FRB for locally Lipschitz continuous operators, in which both the stepsize and the reflection coefficient change across iterations.
		The stepsize is computed in closed form from a simple local Lipschitz estimate, without knowledge of a global Lipschitz constant, while keeping the iteration cost of FRB.
		The analysis of both methods relies on a new Lyapunov function that combines the distance to a solution, a residual, and a gap-related term.
		Alongside nonasymptotic rates, we derive explicit lower bounds on the generated stepsizes, and report numerical experiments on minimax problems, equilibrium models, and regularized regression.
	\end{abstract}



	\section{Introduction}
		This work addresses the monotone inclusion problem
		\begin{equation}\label{eq:P}
		0 \in T(x) \coloneqq F(x)+\partial g(x),
		\end{equation}
		where
		\begin{enumeratass}
		\item \label{ass:f}%
			\(\func{F}{\R^n}{\R^n}\) is a monotone and \emph{locally} Lipschitz-continuous mapping,
		\item \label{ass:g}%
			\(\partial g\) is the subdifferential of a proper convex lsc function \(g:\R^n\to\Rinf\), and
		\item \label{ass:sol}%
			\(\mathcal{X}^\star \coloneqq \zer T\neq\emptyset\).
		\end{enumeratass}
		Inclusion problem \eqref{eq:P} is equivalently stated as the problem of finding \(x^\star\in\dom g\) such that
		\begin{equation}\label{eq:VI}
			g(x) - g(x^\star) + \innprod{F(x^\star)}{x - x^\star} \geq 0
			\qquad\forall x\in \dom g,
		\end{equation}
		and is also referred to as a mixed variational inequality.
		This formulation covers a wide range of problems, including convex minimization, saddle-point problems, complementarity problems, and Nash equilibria \cite{facchinei2003finite}.

		\subsection{Constant stepsize}
			Among the most widely studied methods for solving \eqref{eq:P} are Korpelevich's extragradient method \cite{korpelevich1976extragradient} and Tseng's forward-backward-forward (FBF) method \cite{tseng2000modified}, both handling monotone and Lipschitz continuous \(F\), at the cost of two evaluations of \(F\) per iteration. The extragradient additionally requires two proximal evaluations per step, while Tseng's method requires only one.
			Many other methods following similar extrapolation and reflection principles exist, see \cite{bot2025fast,trandinh2026revisiting} for an overview.
			Of particular relevance to this work is Popov's algorithm \cite{popov1980modification}, rediscovered in the machine learning literature as the \emph{optimistic gradient descent ascent} (OGDA) method, which reduces the per-step cost to one new evaluation of \(F\) by reusing the previous operator value, at the cost of halving the admissible stepsize to \(\gamma < \nicefrac{1}{2L_F}\) where \(L_F\) is the global modulus of Lipschitz continuity of \(F\).
			The \emph{forward-reflected-backward} (FRB) method of \cite{malitsky2020forward} extends Popov's approach to the composite setting for a fixed \(\gamma \in (0,\nicefrac{1}{2L_F})\).

			The present work studies the iteration
			\begin{subequations}\label{eq:FRB-alpha}
				\begin{align}
					\uk*
				={} &
					F(\xk)
					+
					\alpha\bigl(F(\xk)-F(\xk')\bigr)
				\\
					\xk*
				={} &
					\prox_{\gamma g}\bigl(
						\xk
						-
						\gamma\uk*
					\bigr),
				\end{align}
			\end{subequations}
			where \(\alpha\geq0\) is a fixed reflection coefficient.
			The choice \(\alpha=1\) reduces to FRB \cite{malitsky2020forward}, while \(\alpha=0\) reduces to the forward-backward splitting, which is not guaranteed to converge when \(F\) lacks cocoercivity.
			Other generalizations of FRB evaluate \(F\) once at the extrapolated point \(2\xk-\xk'\) \cite{cevher2021reflected}, allow a possibly negative inertial parameter in a Bregman extension for composite minimization \cite{wang2024mirror}, or add a point-inertial term and a third operator value \(F(x^{k-2})\) \cite{soe2025generalized}.
			For globally Lipschitz \(F\), we establish convergence of \eqref{eq:FRB-alpha} for every \(\alpha>\nicefrac12\) with a sufficiently small constant stepsize, and show that this range cannot be enlarged in general.

			The convergence analysis of FRB in \cite{malitsky2020forward} relies on showing that along the iterates the quantity
			\[
				\norm{x^k-x^\star}^2
				+
				2\gamma\innprod{F(x^k)-F(x^{k-1})}{x^\star-x^k}
				+
				\tfrac{1}{2}\norm{x^k-x^{k-1}}^2
			\]
			is nonincreasing for any solution \(x^\star\) of \eqref{eq:P} \cite[Lem. 2.4]{malitsky2020forward}.
			The middle inner product can be negative, and its control is tailored to the unit reflection coefficient.
			We develop a different descent analysis and Lyapunov construction to accommodate general reflection coefficients and adaptive stepsizes.
			In place of this inner product, our construction uses the nonnegative quantity
			\begin{align*}
				\Vk(x)
				& {} \coloneqq
				g(\xk) - g(x) - \innprod{F(\xk)}{x - \xk}\numberthis\label{eq:Vk}\\
				& {} \geq
				g(\xk) - g(x) + \innprod{F(x)}{\xk - x},
			\end{align*}
			defined for \(x \in \dom g\), with the inequality owing to monotonicity of \(F\).
			At \(x=x^\star\in\zer T\) the right-hand side is nonnegative by \eqref{eq:VI}, whence \(\Vk(x^\star)\geq0\).
			On the other hand, while \(\Vk(x^\star)=0\) whenever \(\xk\) solves \eqref{eq:P}, the converse fails in general, so \(\Vk\) is not a merit function.
			This is resolved by taking the supremum of the right-hand side of \eqref{eq:Vk} over a ball of radius \(\delta>0\) around \(x^\star\), resulting in the \emph{restricted gap function} \cite{nesterov2007dual}
			\begin{equation}\label{eq:Gap}
				\Gap_{x^\star,\delta}(y)
			\coloneqq
				\sup_{x\in\dom g\cap\ball{x^\star}{\delta}}
				\bigl\{
					g(y)-g(x)+\innprod{F(x)}{y-x}
				\bigr\},
			\end{equation}
			which is convex and vanishes on \(\ball{x^\star}{\delta}\) only at solutions.
			We prove the rate \(O(\nicefrac{1}{k})\) in this gap along a weighted ergodic average, as well as the best-iterate rate \(\displaystyle\min_{j\leq k}\dist^2(0,T(x^j))=O(\nicefrac{1}{k})\).

			For \(\alpha\in(\nicefrac12,\infty)\), our Lyapunov function is
			\begin{equation}\label{eq:Ukconstant}
				\Uk(x^\star)
			\coloneqq
				\tfrac12\norm{x^k-x^\star}^2
				+
				\tfrac12\norm{x^k-x^{k-1}}^2
				+
				\tfrac{\lambda\gamma^2}{2}\norm{F(x^{k-1})-F(x^{k-2})}^2
				+
				(1+\alpha)\gamma V_{k-1}(x^\star),
			\end{equation}
			for a parameter \(\lambda>0\) fixed as part of the construction.
			\Cref{thm:FRBa} shows that \(\Uk(x^\star)\) is nonincreasing and the iterates converge whenever \(\gamma L_F\) lies below an explicit closed-form threshold, available for every \(\alpha>\nicefrac12\).
			\Cref{ex:sharpness} gives a skew-symmetric linear instance on which the iteration fails to converge for every \(\alpha\in[0,\nicefrac12]\), every stepsize, and every starting point \(x^{-1}=x^0\neq0\).

		\subsection{Adaptive stepsize}
			We next remove the dependence on the global modulus \(L_F\), which may be overly conservative, or nonexistent when \(F\) is only locally Lipschitz.
			Our aim is to generate stepsizes explicitly from information available at the current iterate, with no linesearch and no additional evaluations of \(F\).
			For convex minimization, where \(F=\nabla f\) for a convex function \(f\) with locally Lipschitz gradient, such linesearch-free adaptivity was pioneered by Malitsky and Mishchenko \cite{malitsky2020adaptive}, who proposed the rule
			\begin{equation}\label{eq:adaGD}
				\gamk*
			=
				\min\Bigl\{
					\gamk\sqrt{1+\rhok},\;
					\tfrac{1}{2\Lk}
				\Bigr\},
			\quad
				\xk* = \xk - \gamk*\nabla f(\xk),
			\end{equation}
			based solely on the local Lipschitz estimates
			\begin{equation}\label{eq:Lk}
				\Lk \coloneqq \tfrac{\norm{F(\xk)-F(\xk')}}{\norm{\xk-\xk'}}
			\end{equation}
			from consecutive iterates, with \(\rhok\coloneqq\frac{\gamk}{\gamk'}\) and the convention \(\frac{0}{0}=0\).
			Even when \(F\) is globally Lipschitz continuous, \(\Lk\) can be much smaller than the global Lipschitz modulus, permitting larger stepsizes.
			Rule \eqref{eq:adaGD} was extended to the proximal setting in \cite{latafat2025adaptive}, and to variants with a general growth condition \(\gamk*\leq\gamk\sqrt{1/\alpha+\rhok}\) for \(\alpha>0\) in \cite{malitsky2024proximal,latafat2024convergence,oikonomidis2024universal}.
			Further extensions include acceleration \cite{li2025simple}, stochastic \cite{aujol2025stochastic} and nonconvex settings \cite{yagishita2025linesearch}, Bregman geometries \cite{ou2025linesearch}, and primal-dual splitting schemes \cite{latafat2025adaptive,jang2026alia}.

			We extend \eqref{eq:FRB-alpha} to time-varying stepsizes as
			\begin{subequations}\label{eq:alg-family}
				\begin{align}
					\uk*
				={} &
					F(\xk)
					+
					\alpha\rhok*\bigl(F(\xk)-F(\xk')\bigr)
				\\
				\label{eq:xk*}
					\xk*
				={} &
					\prox_{\gamk* g}\bigl(
						\xk
						-
						\gamk*\uk*
					\bigr),
				\end{align}
			\end{subequations}
			where \(\rhok*\coloneqq\frac{\gamk*}{\gamk}\) scales the reflection term, and the stepsizes \(\gamk*>0\) are chosen explicitly based on past iterates, without requiring a global Lipschitz constant for \(F\) or its knowledge.
			Our simplified scheme \refadaFRB{} (\cref{alg:adaFRB}) fixes \(\alpha\in[1,2]\).
			At \(\alpha=2\) the stepsize rule reads
			\begin{equation}\label{eq:alg-simplified}
				\gamk*
			=
				\min\Bigl\{
					\gamk\sqrt{\tfrac12+\rhok},\;
					\tfrac1{5\Lk}
				\Bigr\}.
			\end{equation}
			The first term limits how fast the stepsize may grow, while the second, which we call the safety bound, caps it at a multiple \(\frac{c}{\Lk}\) of the inverse local Lipschitz estimate, with \(c=\nicefrac15\) at \(\alpha=2\), its largest value over \(\alpha\in[1,2]\).
			The rule for general \(\alpha\in[1,2]\) follows the same pattern, with an additional cap \(\rhok*\leq\frac23+\frac{2\alpha}{5}\) that is never active at \(\alpha=2\), and is given in full in \cref{alg:adaFRB}.
			Its first term \(\gamk\sqrt{\nicefrac{1}{\alpha}+\rhok}\) has the form of the growth condition of \cite{malitsky2024proximal,latafat2024convergence,oikonomidis2024universal}, with \(\alpha\) the reflection coefficient, and reduces to that of \eqref{eq:adaGD} at \(\alpha=1\).
			In a steep or ill-conditioned region \(\Lk\) is large, and the safety bound \(\nicefrac{c}{\Lk}\) forces the stepsize down. Once the iterates leave such a region and the safety bound becomes inactive, the first term lets the stepsize start increasing again within at most two updates.\footnote{%
				This owes to the fact that \(\sqrt{\nicefrac{1}{\alpha}+\sqrt{\nicefrac{1}{\alpha}+\rhok'}}>1\) for any \(\rhok'>0\) and any \(\alpha\in[1,2]\).
			}

			A sharper variant exploits the inner product \(\innprod{F(x^k)-F(x^{k-1})}{F(x^{k-1})-F(x^{k-2})}\), which enters the descent inequality behind \eqref{eq:alg-simplified} and is already available at iteration \(k\).
			Folding a tighter Young-type estimate of it into the stepsize (see \cref{thm:Young}) enlarges the safety bound, at the cost of a more involved update.
			We present the case \(\alpha=2\) in \cref{alg:adaFRB+2}, and leave the general scheme inclusive of other values of \(\alpha\) to the appendix.
			As at constant stepsize, both schemes attain the best-iterate rate \(\min_{j\leq k}\dist^2(0,T(x^j))=O(\nicefrac{1}{k})\) and the ergodic rate \(O(\nicefrac{1}{k})\) in the restricted gap (\cref{thm:adaFRB}).

			Our safety coefficient \(c=\frac{1}{7-\alpha}\) is smaller than the coefficient \(\nicefrac12\) of \eqref{eq:adaGD}.
			This is to be expected.
			A similar loss is already present at fixed stepsize, where the FRB condition \(\gamma<\nicefrac{1}{2L_F}\) is one quarter of the proximal gradient bound \(\gamma<\nicefrac{2}{L_F}\).
			Within our analysis, the difference stems from the absence of a cost function.
			In the Lyapunov analysis of \eqref{eq:adaGD} and its extensions, a key nonnegative term is \(f(\xk) - f(x^\star)\) (or \((f+g)(\xk)-(f+g)(x^\star)\) in the composite case), which allows the control of certain inner products using subgradient inequalities.
			No such quantity exists for general monotone \(F\).
			Instead its role is taken by \(\Vk(x^\star)\) from \eqref{eq:Vk}.
			Our coefficient is also smaller than the constant stepsize threshold of \cref{thm:FRBa}, mostly because the analysis must allow the stepsize to increase between iterations.
			At \(\alpha=2\), the sharper variant recovers this threshold.

			Compared with minimization, linesearch-free adaptivity for variational inequalities and monotone inclusions is less developed.
			The GRAAL and aGRAAL methods \cite{malitsky2020golden,alacaoglu2023beyond} are the closest prior work in this setting.
			Like \eqref{eq:alg-family}, they require a single evaluation of \(F\) and one proximal step per iteration, but build the iterate around an auxiliary averaged point \(\bar x^k\),
			\begin{equation}\label{eq:aGRAAL}
				\bar x^k
			=
				\tfrac{\varphi-1}{\varphi}\xk + \tfrac{1}{\varphi}\bar x^{k-1},
			\qquad
				\xk*
			=
				\prox_{\gamk* g}\bigl(\bar x^k - \gamk* F(\xk)\bigr),
			\end{equation}
			with \(\varphi\in\bigl(1,\tfrac{1+\sqrt{5}}{2}\bigr]\) and adaptive stepsize \(\gamk*=\min\bigl\{\omega\gamk,\,\tfrac{\varphi^2}{4\gamk'\Lk^2},\,\gamma_{\rm max}\bigr\}\), where \(1\leq\omega\leq\tfrac{1}{\varphi}+\tfrac{1}{\varphi^2}\) and \(\gamma_{\rm max}>0\) is a prescribed upper bound, shown in \cite{alacaoglu2023beyond} to be removable when \(g=\delta_C\).
			Both the stepsize update and the analysis are fundamentally different from ours.
			In particular, the bound \(\frac{\varphi^2}{4\gamk'\Lk^2}\) couples the previous stepsize to the square of \(\Lk\), contrasting with our safety bound \(\tfrac{c}{\Lk}\) (e.g.\ \(\nicefrac{1}{5\Lk}\) at \(\alpha=2\)), which is linear in \(1/\Lk\) and involves no past stepsize.
			We compare against the adaptive variant aGRAAL in our simulations.

			Accelerated rates for \eqref{eq:P} have also been obtained via Halpern-type anchoring \cite{yoon2021accelerated} or vanishing-damping dynamics \cite{bot2025fast}, both of which assume a known global Lipschitz constant \(L_F\).
			Our focus here is complementary, targeting adaptivity to the local Lipschitz geometry of \(F\) without any such knowledge.

			In summary, our main contributions are (i) the convergence of \eqref{eq:FRB-alpha} at constant stepsize for every \(\alpha>\nicefrac12\), a range that cannot be enlarged, and (ii) linesearch-free adaptive stepsizes for locally Lipschitz \(F\), which need no upper bound and stay bounded away from zero.
			Both results, together with \(O(\nicefrac{1}{k})\) rates, are stated in \cref{sec:algorithms} and proved in \cref{sec:convergence}.
			\Cref{sec:simulations} reports numerical experiments.

		\subsection{Notation}
			We denote by \(\innprod{\cdot}{\cdot}\) and \(\norm{\cdot}\) the Euclidean inner product and norm on \(\R^n\), by \(\Rinf\) the extended real line, and by \(\N\) the set of natural numbers including zero.
			The closed ball of radius \(\delta>0\) around \(x\) is \(\ball{x}{\delta}\coloneqq\{y\in\R^n\mid\norm{y-x}\leq\delta\}\).
			For a proper convex lsc function \(g\), \(\dom g\) and \(\partial g\) denote its domain and subdifferential, and \(\prox_{\gamma g}(x)\coloneqq\argmin_w\{g(w)+\frac{1}{2\gamma}\norm{w-x}^2\}\) its proximal mapping with stepsize \(\gamma>0\).
			For an operator \(T\), \(\zer T\coloneqq\{x\mid0\in T(x)\}\) is its set of zeros, and \(\dist(0,T(x))\coloneqq\inf_{v\in T(x)}\norm{v}\).
			Finally, \(\delta_C\) is the indicator function of a set \(C\), and \([t]_+\coloneqq\max\{t,0\}\) and \([t]_-\coloneqq\max\{-t,0\}\) denote the positive and negative part of \(t\in\R\).

	\section{Proposed algorithms}\label{sec:algorithms}
		\subsection{The globally Lipschitz case: a generalized FRB splitting}\label{sec:alg_constant}
			We begin by proposing a linesearch-free method corresponding to the family of algorithms \eqref{eq:alg-family} with constant stepsizes \(\gamk\equiv\gamma\), hence \(\rhok\equiv1\), under the assumption that \(F\) is globally Lipschitz continuous with known modulus \(L_F\).
			The full generality of mere local Lipschitzianity (or agnosticism to Lipschitz moduli) will be addressed in the next subsection.
			With these simplifications, the iterates have the form
			\begin{equation*}
				\tag{FRB\ensuremath{_\alpha}}\label{eq:FRBa}
				\xk*
			=
				\prox_{\gamma g}\Bigl(
					\xk
					-
					\gamma\Bigl[
						(1+\alpha)F(\xk)
						-
						\alpha F(\xk')
					\Bigr]
				\Bigr),
			\end{equation*}
			for some stepsize \(\gamma\) and some parameter \(\alpha\).
			We show that any \(\alpha>\nicefrac{1}{2}\) works, provided that the stepsize satisfies
			\begin{equation}\label{eq:FRBa_gamma}
				\gamma L_F
			<
				c(\alpha)
			\coloneqq
				\frac{2\alpha-1}{\alpha\abs{2-\alpha}+\sqrt{\alpha^2(2-\alpha)^2+2(2\alpha-1)(2\alpha+1)^2}}.
			\end{equation}
			Here, note that for \(\alpha=1\) one has \(c(1)=\tfrac{\sqrt{19}-1}{18}\approx0.187\), leading to a bound on the stepsize more conservative than \(\gamma L_F<\nicefrac{1}{2}\) of \cite{malitsky2020forward}.
			This conservatism is a natural consequence of our more general analysis, which is designed to accommodate the absence of global Lipschitz continuity and recovers the constant-stepsize setting as a worst-case specialization.
			The same happens for the (proximal) gradient method, that is, \eqref{eq:P} with \(F=\nabla f\).
			The adaptive analyses of \cite{malitsky2024proximal,latafat2024convergence}, specialized to a constant stepsize, yield \(\gamma L_F<\nicefrac{\sqrt{2}}{2}\), whereas the classical analysis allows the tight bound \(\gamma L_F<2\).

			The value of \(c(\alpha)\) peaks exactly at \(\alpha=2\), while \(c(\alpha)\sim\nicefrac{1}{\alpha}\) as \(\alpha\to\infty\).
			Note that this quantity is well defined and positive precisely when \(\alpha>\nicefrac{1}{2}\).
			This restriction is not an artifact of our analysis, since the range of admissible \(\alpha\) is \emph{tight}, as the next counterexample shows.
			The details are deferred to \cref{proof:ex:sharpness}.

			\begin{example}[sharpness of \(\alpha>1/2\)]\label{ex:sharpness}%
				Let \(g\equiv0\) and \(F(x)=Sx\) with \(S=\left(\begin{smallmatrix}0&1\\-1&0\end{smallmatrix}\right)\).
				Then, \(F\) is monotone, \(1\)-Lipschitz, with unique solution \(x^\star=0\).
				The constant-stepsize iteration is then a linear recursion, and the roots of its characteristic polynomial lie strictly inside the unit disk, hence \(\xk\to x^\star\), if and only if \(\alpha>\nicefrac{1}{2}\) and \(\gamma^2<\frac{2\alpha-1}{\alpha^2(1+2\alpha)}\).
				Hence, for \(\alpha\leq\nicefrac{1}{2}\) the iteration initialized at \(x^{-1}=x^0\ne0\) fails to converge to \(x^\star\) for every \(\gamma>0\).
				For \(\alpha>\nicefrac{1}{2}\) stability persists however large \(\alpha\) is, with the admissible \(\gamma\) shrinking like \(\nicefrac{1}{\alpha}\), consistent with \(c(\alpha)\sim\nicefrac{1}{\alpha}\) above.
			\end{example}

			\begin{theorem}[convergence at constant stepsize]\label{thm:FRBa}%
				Suppose that \(F\) is globally Lipschitz continuous with modulus \(L_F\), and let \(\seq{\xk}\) be generated by the generalized FRB iterations \eqref{eq:FRBa} with \(\alpha>\nicefrac{1}{2}\) and \(\gamma<\nicefrac{c(\alpha)}{L_F}\) as in \eqref{eq:FRBa_gamma}.
				Then, \(\seq{\xk}\) converges to a solution of \eqref{eq:P}.
				Moreover, for any solution \(x^\star\) of \eqref{eq:P} and \(k\geq 1\) it holds that
				\begin{align*}
					\Uk*(x^\star)
				\leq
					\Uk(x^\star)
				&
					-
					\gamma
					\Vk'(x^\star)
					-
					C
					\bigl(
						c(\alpha)^2
						-
						(\gamma L_F)^2
					\bigr)
					\norm{\xk-\xk'}^2,
				\end{align*}
				where \(\Uk\) is as in \eqref{eq:Ukconstant} with \(\lambda=2\alpha(1+2\alpha)+\nicefrac{\alpha\abs{2-\alpha}}{c(\alpha)}\), and
				\(
					C
				\coloneqq
					(1+2\alpha)^2
					+
					\tfrac{\alpha\abs{2-\alpha}}{2c(\alpha)}
				>
					0
				\).
			\end{theorem}

			This result is a special case of the more general adaptive analysis.
			Accordingly, the ergodic and best-so-far rates established for the adaptive schemes in \cref{thm:adaFRB} hold for \eqref{eq:FRBa} as well, with \(\rhok\equiv1\) in the ergodic average \eqref{eq:barxk}.

		\subsection{The general case: the adaptive \texorpdfstring{\refadaFRB}{adaFRBα}}
			We now derive a family of adaptive schemes for the iteration family \eqref{eq:alg-family} that can cope with mere \emph{local} Lipschitzianity of \(F\).
			The key ingredient is the one-step inequality of \cref{thm:main:innprods} below, from which all stepsize rules in this work are derived.
			It applies to any sequence recursively generated via the update in \eqref{eq:alg-family}, where the reflection coefficient is the stepsize ratio scaled by \(\alpha\).
			This proportionality simplifies the analysis, as will become apparent below.
			The proof of the lemma, deferred to \cref{proof:thm:main:innprods}, first establishes the more general inequality \eqref{eq:main:general} for an arbitrary reflection coefficient, and then specializes it to iterates of the form \eqref{eq:alg-family} investigated here.

			The inequality adopts the shorthand notation
			\begin{equation}\label{eq:Dk}
				\Dxk
			\coloneqq
				\xk-\xk'
			\quad\text{and}\quad
				\DFk
			\coloneqq
				F(\xk)-F(\xk').
			\end{equation}
			In this notation, the local Lipschitz estimate \eqref{eq:Lk} reads \(\Lk=\frac{\norm{\DFk}}{\norm{\Dxk}}\), with the convention \(\frac{0}{0}=0\).

			\begin{lemma}[key inequality]\label{thm:main:innprods}%
				Consider the iterates \eqref{eq:alg-family} for some \(\alpha\geq0\).
				Then, with \(\Lk\) as in \eqref{eq:Lk}, \(\Dxk\) and \(\DFk\) as in \eqref{eq:Dk}, and \(\Vk\) as in \eqref{eq:Vk}, for any \(k\geq1\) and \(x\in\dom g\) it holds that
				\begin{align*}
				&
					\tfrac{1}{2}
					\norm{\xk*-x}^2
					+
					\tfrac{1}{2}
					\norm{\Dxk*}^2
			+
					\gamk*(1+\alpha\rhok*)\Vk(x)
				\\
				\leq{} &
					\tfrac{1}{2}
					\norm{\xk-x}^2
					+
					\alpha\rhok*\gamk*\Vk'(x)
					+
					\alpha^2\rhok^2\gamk*^2
					\norm{\DFk'}^2
				\\
				&
					+
					\rhok*^2
					\biggl[
						(1+\alpha\rhok*)^2\gamk^2\Lk^2+1
						-
						\alpha
					\biggr]
					\norm{\Dxk}^2
				\\
				&
					+
					\alpha\rhok\gamk\rhok*^2(2-\alpha)
					\innprod{\Dxk}{\DFk'}
					-
					2\alpha\rhok(1+\alpha\rhok*)\gamk^2
					\rhok*^2
					\innprod{\DFk}{\DFk'}.
				\numberthis\label{eq:main}
				\end{align*}
			\end{lemma}

		\subsubsection{\texorpdfstring{\refAdaFRB}{AdaFRBα}}
			\begin{algorithm}[tb]%
				\caption{\AdaFRB}%
				\label{alg:adaFRB}%
				\begin{algorithmic}[1]
				\itemsep=5pt
				\item[{%
					Choose \(\alpha\in[1,2]\), \(x^0\in\dom g\), \(\gamma_0>0\), and \(L_0>0\)%
				}]
				\item[Set \(x^{-1}=x^0\), \(\rho_0=1\), \(c=\frac{1}{7-\alpha}\)]
				\item[Repeat for \(k=0,1,\dots\)]
				\State \label{state:adaFRB:rhok*}%
					\(
						\gamk*
					=
						\min\left\{
							\gamk
							\sqrt{\tfrac{1}{\alpha}+\rhok}
						,\,
							\left(
								\tfrac{2}{3}+\tfrac{2\alpha}{5}
							\right)
							\gamk
						,\,
							\frac{c}{\Lk}
						\right\}
					\)
					~and~
					\(
						\rhok*=\frac{\gamk*}{\gamk}
					\)
				\Comment%
					\(\Lk\) as in \eqref{eq:Lk}

				\State
					\(
						\uk*
					=
						F(\xk)
						+
						\alpha\rhok*
						(F(\xk)-F(\xk'))
					\)

				\State
					\(
						\xk*
					=
						\prox_{\gamk* g}\bigl(
							\xk
							-
							\gamk*\uk*
						\bigr)
					\)
				\end{algorithmic}
			\end{algorithm}

			In \refadaFRB{} the parameter \(\alpha\in[1,2]\) is user-specified.
			While any \(\alpha>\nicefrac{1}{2}\) could in principle be considered, values \(\alpha>2\) would introduce case distinctions depending on the sign of \(2-\alpha\), which we avoid for clarity of presentation.
			Our convergence analysis revolves around the derivation of a merit function \(U_k(x^\star)\geq0\) for each solution \(x^\star\in\zer T\), where the selection of the next stepsize \(\gamk*\) is made in such a way to ensure a decrease condition \(U_{k+1}(x^\star)\leq U_k(x^\star)-a\norm{\Dxk}^2\) for some \(a>0\), its explicit form being given in \eqref{eq:Uk} below.
			Because of the ``asynchrony'' of the arguments of the inner products within \cref{thm:main:innprods} we cannot rely on, say, monotonicity of \(F\) to infer their sign, which is why we shall convert them into squared norms.
			The family of adaptive algorithms \refadaFRB{} is obtained by bounding these inner products via Young's inequality.

			While conservative, this choice allows for a conveniently simple update rule, outlined in \cref{alg:adaFRB}.
			Less conservative bounds are discussed in the next subsection, leading to theoretically larger stepsizes at the expense of a more involved update rule.

			The following result covers both \refadaFRB{} and its refinement \refadaFRBII, introduced in the next subsection.
			For \(k\geq1\) and \(x\in\dom g\), the merit function of both schemes is
			\begin{align}
			\label{eq:Uk}
				\Uk(x)
			\coloneqq{} &
				\tfrac{1}{2}
				\norm{\xk-x}^2
				+
				\tfrac{1}{2}
				\norm{\xk-\xk'}^2
				+
				\tfrac{\lambda(\gamk\rhok)^2}{2}
				\norm{F(\xk')-F(\xk'')}^2
			\\
			\nonumber
			&
				+
				(1+\alpha\rhok)
				\gamk\Vk'(x),
			\end{align}
			with \(\lambda\) as in \eqref{eq:adaFRB:coeff} for \refadaFRB{} and \(\lambda=20\) for \refadaFRBII.

			\begin{theorem}[convergence of {\protect\refadaFRB} and {\protect\refadaFRBII}]\label{thm:adaFRB}%
				Let \cref{ass:f,ass:g,ass:sol} hold, and consider the iterates generated by either \refadaFRB{} (\cref{alg:adaFRB})
				or \refadaFRBII{} (\cref{alg:adaFRB+2}), with \(\alpha\) and \(c\) as in the respective algorithm.
				Then, the sequence \(\seq{\xk}\) converges to a solution.

				Moreover, with \(\Uk\) as in \eqref{eq:Uk}, for any solution \(x^\star\in\zer T\) and \(k\geq 1\) the following hold:
				\begin{enumerate}
				\item \label{thm:adaFRB:Lyapunov}%
					\(0\leq\Uk*(x^\star)\leq\Uk(x^\star)\).

				\item \label{thm:adaFRB:gammin}%
					\(\gamk\geq\frac{1}{\sqrt{\alpha}}\min\set{\gamma_0,\frac{c}{L_{F,\V}}}\), where \(L_{F,\V}\) is a Lipschitz modulus of \(F\) on a compact convex set \(\V\) containing the iterates.

				\item \label{thm:adaFRB:residual}%
					\(
						\min_{1\leq j\leq k}
						\norm{x^j-x^{j-1}}^2
					\leq
						\tfrac{\Uk_1(x^\star)}{C_{\rm min}k}
					\)
					for some \(C_{\rm min}>0\) and \(\min_{j\leq k}\dist^2(0,T(x^j))\sim O(\nicefrac{1}{k})\).%

				\item \label{thm:adaFRB:ergodic}%
					For any \(\delta\geq\sqrt{2\Uk_1(x^\star)}\), the rate
					\(
						\Gap_{x^\star,\delta}(\bar x^k)
					\leq
						\frac{
							\sup\set{
								\Uk_2(x)
							}[
								x\in\dom g\cap \ball{x^\star}{\delta}
							]
						}{
							\sum_{j=2}^{k+1}\gamma_j
						}
					\)
					holds\footnote{%
						Note that the numerator is finite, owing to the fact that \(-g\) (hence \(\Vk_1\), and in turn \(\Uk_2\) too) is upper bounded on bounded sets by virtue of properness and lower semicontinuity of \(g\) \cite[Cor. 1.10]{rockafellar1998variational}.
					}
					for the ergodic sequence \(\seq{\bar x^k}\subset\dom g\) given by
					\begin{equation}\label{eq:barxk}
						\bar x^k
					\coloneqq
						\frac{
							(1+\alpha\rhok*)\gamk*\,\xk
							+
							\sum_{j=2}^k
							\alpha\gamma_j(\frac{1}{\alpha}+\rho_j-\rho_{j+1}^2)\,x^{j-1}
						}{
							(1+\alpha\rhok*)\gamk*
							+
							\sum_{j=2}^k
							\alpha\gamma_j(\frac{1}{\alpha}+\rho_j-\rho_{j+1}^2)
						}.
					\end{equation}
				\end{enumerate}
			\end{theorem}

			\begin{remark}[average stepsize]\label{rem:gamavg}%
					\Cref{thm:adaFRB:ergodic} implies an \(O(\nicefrac{1}{k})\) rate for the ergodic sequence \(\seq{\bar x^k}\), since \Cref{thm:adaFRB:gammin} gives \(\sum_{j=2}^{k+1}\gamma_j\geq\frac{k}{\sqrt{\alpha}}\min\set{\gamma_0,\frac{c}{L_{F,\V}}}\).
				For \refadaFRB{} the factor \(\nicefrac{1}{\sqrt{\alpha}}\) can in fact be dropped for the average stepsize.
				Indeed, since \(\frac{2}{3}+\frac{2\alpha}{5}\geq\sqrt{\alpha}\) for every \(\alpha\in[1,2]\), the second term at \cref{state:adaFRB:rhok*} is at least \(\gamk\sqrt{\alpha}\), and \cref{thm:gamavg} gives
				\[
					\textstyle\sum_{j=2}^{k+1}\gamma_j
				\geq
					k\min\set{\gamma_0,\tfrac{c}{L_{F,\V}}}
				\qquad
					\forall k\in\N.
				\]
				The same argument fails for the tightened variant \refadaFRBII{} presented in the next subsection, since an adaptive upper cap \(\betk\) on the ratio can drop as low as \(\betk=1\) when consecutive differences \(\DFk\) and \(\DFk'\) are negatively aligned (i.e., when \([\cos\phi_k]_-=1\) in \cref{state:adaFRB+2:betak}).
			\end{remark}

		\subsubsection{\texorpdfstring{\refAdaFRBII}{AdaFRB+2}: improved stepsizes via a tighter Young's inequality}
			The stepsize condition in \refadaFRB{} arises from bounding the two inner product terms in \cref{thm:main:innprods} via Young's inequality, which has the advantage of requiring no knowledge of these quantities at runtime.
			However, in our setting both inner products are explicitly available when the stepsize is computed, and this additional information can be leveraged to obtain tighter estimates.
			The following lemma makes this precise, and leads to the refined \refadaFRBII.
			This variant has an improved theoretical lower bound on the stepsize sequence, while its practical impact is modest, as reported in our numerical experiments.
			We nonetheless include it as a theoretical refinement.

			\begin{lemma}[tighter Young's (in)equalities]\label{thm:Young}%
				For any \(a,b\in\R^n\setminus\set{0}\) and \(\varepsilon>0\) one has that
				\begin{gather}\label{eq:Youngtight}
					\innprod{a}{b}
				=
					\overbracket[0.5pt]{
						\cos\theta_{a,b}
						\vphantom{\tfrac{\norm{a}}{\norm{b}}}
					}^{\in[-1,1]}
					\,
					\overbracket[0.5pt]{
						\sigma\bigl(\varepsilon\tfrac{\norm{a}}{\norm{b}}\bigr)
					}^{\in(0,1]}
					\,
					\Bigl(
						\tfrac{\varepsilon}{2}\norm{a}^2
						+
						\tfrac{1}{2\varepsilon}\norm{b}^2
					\Bigr)
				\intertext{%
					where
					\(
						\sigma(t)
					\coloneqq
						\bigl(\frac{t}{2}+\frac{1}{2t}\bigr)^{-1}
					\in
						(0,1]
					\)
					for \(t>0\) and
					\(
						\cos\theta_{a,b}\coloneqq\tfrac{\innprod{a}{b}}{\norm{a}\norm{b}}\in[-1,1]
					\).
					In particular,
				}
				\label{eq:Youngcos_pos}
					-[\cos\theta_{a,b}]_-
					\left(
						\tfrac{\varepsilon}{2}\norm{a}^2
						+
						\tfrac{1}{2\varepsilon}\norm{b}^2
					\right)
				\leq
					\innprod{a}{b}
				\leq
					[\cos\theta_{a,b}]_+
					\left(
						\tfrac{\varepsilon}{2}\norm{a}^2
						+
						\tfrac{1}{2\varepsilon}\norm{b}^2
					\right)
				\quad
					\forall\varepsilon>0.
				\end{gather}
			\end{lemma}

			Identity \eqref{eq:Youngtight} follows from \(\innprod{a}{b}=\norm{a}\norm{b}\cos\theta_{a,b}\) and
			\[
				\norm{a}\norm{b}
			=
				\sigma\bigl(\varepsilon\tfrac{\norm{a}}{\norm{b}}\bigr)
				\left[
					\tfrac{\varepsilon}{2}\norm{a}^2
					+
					\tfrac{1}{2\varepsilon}\norm{b}^2
				\right]
			\]
			for any \(\varepsilon>0\), and separates the angular factor \(\cos\theta_{a,b}\), independent of \(\varepsilon\), from the factor \(\sigma\bigl(\varepsilon\tfrac{\norm{a}}{\norm{b}}\bigr)\), independent of the relative orientation.
			The standard Young's inequality is recovered by replacing both \(\cos\) and \(\sigma\) by their uniform upper bound \(1\).

			For simplicity we restrict to \(\alpha=2\), for which the first inner product in \cref{thm:main:innprods} vanishes and only \(\innprod{\DFk}{\DFk'}\) remains.
			Its value is known at iteration \(k\), but the Young parameter that the analysis pairs it with depends on the next stepsize through \(\rhok*\).
			We therefore use \eqref{eq:Youngcos_pos}, which retains the available angular factor \([\cos\phi_k]_-\) and bounds the \(\sigma\) factor by \(1\).%

			\begin{algorithm}[tb]
				\caption{\AdaFRB+_2}
				\label{alg:adaFRB+2}%
				\begin{algorithmic}[1]
				\itemsep=5pt
				\item[{%
					Choose \(x^0\in\dom g\), \(\gamma_0>0\), \(L_0>0\), and \(c\in\bigl(0,\tfrac{\sqrt{6}}{10}\bigr)\)%
				}]
				{\Comment%
					\(
						\tfrac{\sqrt{6}}{10}
					\approx
						0.245
					\)
				}
				\item[{%
					Set \(\alpha=2\), \(x^{-2}=x^{-1}=x^0\) and \(\rho_0=1\)%
				}]
				\item[Repeat for \(k=0,1,\dots\)]
				\State \label{state:adaFRB+2:betak}%
					\(
					\betk=\sqrt{\tfrac{5}{2+3[\cos\phi_k]_-}}
					\)
					~where~
					\(
						[\cos\phi_k]_-
					=
						\frac{
							\left[
								\innprod{\DFk}{\DFk'}
							\right]_-
						}{
							\norm{\DFk}
							\norm{\DFk'}
						}
					\)
				\Comment
					\(\DFk\) as in \eqref{eq:Dk}

				\State \label{state:adaFRB+2:gamk*}%
					\(
						\gamk*
					=
						\min\left\{
							\gamk
							\sqrt{
								\tfrac{1}{\alpha}+\rhok
							}
						,\,
							\gamk
							\betk
						,\,
							\frac{c}{\Lk}
						\right\}
					\)
					~and~
					\(\rhok*=\frac{\gamk*}{\gamk}\)
				\Comment%
					\(\Lk\) as in \eqref{eq:Lk}%

				\State
					\(
						\uk*
					=
						F(\xk)
						+
						\alpha\rhok*
						(F(\xk)-F(\xk'))
					\)

				\State
					\(
						\xk*
					=
						\prox_{\gamk* g}\bigl(
							\xk
							-
							\gamk*\uk*
						\bigr)
					\)
				\end{algorithmic}
			\end{algorithm}

			The resulting update rule is presented in \cref{alg:adaFRB+2}, which is the scheme we analyze in the rest of this work alongside \refadaFRB.
			The gain over \refadaFRB{} lies in the cap on \(\gamk*\Lk\), which at \(\alpha=2\) rises from \(c=\frac{1}{5}\) to any \(c<\nicefrac{\sqrt{6}}{10}\approx0.245\), the largest cap the analysis affords.
			This limiting value is exactly \(c(2)\) in \eqref{eq:FRBa_gamma}, so the adaptive rule recovers the sharp constant-stepsize threshold of \cref{thm:FRBa} at \(\alpha=2\).
			The price is the more intricate update at \cref{state:adaFRB+2:betak}, and a ratio bound \(\gamk*\leq\betk\gamk\) that may drop to \(\betk=1\) in the worst case, that is, when \(\DFk\) and \(\DFk'\) point in opposite directions, against the constant \(\betk\equiv\frac{2}{3}+\frac{2\alpha}{5}>1\) of \refadaFRB.
			\Cref{fig:betak} shows \(\betk\) on four problems from \cref{sec:simulations}.
			On some of them \(\betk\) stays at its largest value \(\sqrt{\nicefrac{5}{2}}\), while on others it is mostly close to \(1\) and thus below the constant cap \(\frac{2}{3}+\frac{4}{5}\) of \refadaFRB_2.
			This is consistent with the modest practical gain of \refadaFRBII{} observed in \cref{sec:simulations}.
			The case \(\alpha\in[1,2)\) brings the same kind of theoretical improvement through a more involved update rule, also depending on the lossless decomposition \eqref{eq:Youngtight} of the first inner product.
			As a result, even if \([\cos\phi_k]_-=1\), \(\betk\) is unlikely to reach its minimum value 1, since this would also require the other inner product to be unfavorable.
			This general case covering all values of \(\alpha\in[1,2]\) is presented in \cref{sec:adaFRB+}.

			\begin{remark}[limiting cap]\label{rem:adaFRB+2:limit}%
				For \refadaFRBII{}, the analysis requires \(c<\frac{\sqrt{6}}{10}\), since the proof of \cref{thm:adaFRB} yields a descent inequality whose slack is proportional to \(\frac{6}{100}-c^2\), which vanishes at the limiting value \(c=\frac{\sqrt{6}}{10}\).
				At this value the descent is only nonstrict, so the ergodic rate of \cref{thm:adaFRB:ergodic} and the fact that its cluster points solve \eqref{eq:P} would still follow, whereas convergence of the whole sequence and the residual rate of \cref{thm:adaFRB:residual} rely on the slack being strictly positive and would no longer be guaranteed by the present analysis.
				Since the gain from allowing the limiting value is marginal, \cref{alg:adaFRB+2} restricts to \(c<\frac{\sqrt{6}}{10}\).
			\end{remark}

	\section{Convergence analysis}\label{sec:convergence}
		Throughout the analysis, each iterate \(\xk\) with \(k\geq1\) is associated with the subgradient
		\begin{equation}\label{eq:tildeg}
			\nabla* g(\xk)
		\coloneqq
			-\tfrac{1}{\gamk}\Dxk-\uk
		\in
			\partial g(\xk),
		\end{equation}
		where the inclusion follows from the definition of \(\xk\) in \eqref{eq:xk*} and the optimality condition of the proximal mapping.

		\begin{figure}[t]
			\begin{tikzpicture}

				\definecolor{glacier}{HTML}{76C7C0}
				\definecolor{steelblue}{HTML}{4682B4}
				\definecolor{deepteal}{HTML}{008080}

				\pgfplotsset{%
					betaline/.style = {%
						line width = 0.75pt,
					},
					shift y/.style = {
						y filter/.expression = {y>=\ymax ? y + \shiftgap*(#1) : (y <= \ymin ? y + \shiftgap*(#1) : y)},
					},
					linCN/.style = {betaline, MidnightBlue},
					nonlinCN/.style = {betaline, MidnightBlue!50, densely dashed, shift y = -1},
					a9a/.style = {betaline, teal, mark=o, mark repeat=25, shift y = -2},
					minmax/.style = {betaline, Violet, shift y = +1},
				}

				\def\shiftgap{0.01}

				\begin{axis}[
					name = cosphik,
					width = 0.4\textwidth,
					height = 0.2\textheight,
					myaxis,
					title = {\([\cos\phi_k]_-\)},
					title style = {
						yshift = -0.75em,
					},
					legend style = {
						at = {(1, 1.3)},
					},
					legend columns = 4,
					xmin = 0,
					xmax = 200,
					ymin = -0.1,
					ymax = 1.1,
				]
					\def\ymax{2}
					\def\ymin{0}

					\addplot[linCN] table {TeX/Tikz/betak_plots/cournot-nash_equilibrium_problem/n100/cosphi_minus.txt};
					\addlegendentry{lin. CN\(_{n=100}\) (\S\ref{sec:CNlinear})}

					\addplot[nonlinCN] table {TeX/Tikz/betak_plots/nonlinear_cournot-nash_equilibrium_problem/n1000_scenario1/cosphi_minus.txt};
					\addlegendentry{nonlin. CN\(_{\#1}\) (\S\ref{sec:CNnonlinear})}

					\addplot[minmax] table {TeX/Tikz/betak_plots/bilinear_matrix_game_control/n500/cosphi_minus.txt};
					\addlegendentry{minmax\(_{n=500}\) (\S\ref{sec:bilinear_game})}

					\addplot[a9a] table {TeX/Tikz/betak_plots/logistic_regression_problem/a9a_100/cosphi_minus.txt};
					\addlegendentry{a9a\(_{n=100}\) (\S\ref{sec:logreg})}
				\end{axis}

				\begin{axis}[
					at = {(cosphik.east)},
					anchor = west,
					xshift = 1.5cm,
					width = 0.4\textwidth,
					height = 0.2\textheight,
					myaxis,
					title = {\(\beta_k=\sqrt{\frac{5}{2+3[\cos\phi_k]_-}}\)},
					title style = {
						yshift = -0.75em,
					},
					xmin = 0,
					xmax = 200,
					ymin = 0.9,
					ymax = sqrt(5/2)+0.1,
					ytick = {1, 22/15, sqrt(5/2)},
					yticklabels = {1,\(\nicefrac{22}{15}\),\(\sqrt{\nicefrac{5}{2}}\)},
				]
					\addplot[black!60, densely dashed, line width=0.5pt, forget plot, domain=0:200, samples=2] {22/15};
					\def\ymax{sqrt(5/2)}
					\def\ymin{0}

					\addplot[linCN] table {TeX/Tikz/betak_plots/cournot-nash_equilibrium_problem/n100/betak.txt};

					\addplot[nonlinCN] table {TeX/Tikz/betak_plots/nonlinear_cournot-nash_equilibrium_problem/n1000_scenario1/betak.txt};

					\addplot[a9a] table {TeX/Tikz/betak_plots/logistic_regression_problem/a9a_100/betak.txt};

					\addplot[minmax] table {TeX/Tikz/betak_plots/bilinear_matrix_game_control/n500/betak.txt};
				\end{axis}
			\end{tikzpicture}
			\caption{%
				Values of \([\cos\phi_k]_-\) (left) and \(\betk\) (right) in \refadaFRBII{} over the first 200 iterations on four problems from \cref{sec:simulations}.
				Curves at the extreme values \(0\) and \(\sqrt{\nicefrac{5}{2}}\) are slightly offset for visibility.
				The value \(\nicefrac{22}{15}\) is the constant cap \(\frac{2}{3}+\frac{4}{5}\) of \refadaFRB_2.%
			}%
			\label{fig:betak}%
		\end{figure}

		Our unified convergence analysis, encompassing both the static and adaptive variants, revolves around the identification of a Lyapunov function for the iterates generated by the proposed algorithms.
		The inequality in \cref{thm:main:innprods} provides a natural starting point, as it already features squared norms and \(\Vk\).
		It remains to incorporate the two inner-product terms into the same quadratic structure, which can be done via \cref{thm:Young}.
		\begin{subequations}\label{eq:taukcosphik}
			The first inner product is handled losslessly with the Young parameter \(\nicefrac{\varepsilon}{\gamk\rhok}\) for some \(\varepsilon>0\), namely
			\begin{equation}\label{eq:tau-decomposition}
				\innprod{\Dxk}{\DFk'}
			=
				\tfrac{\varepsilon\tauk}{2\gamk\rhok}\norm{\Dxk}^2
				+
				\tfrac{\rhok\gamk\tauk}{2\varepsilon}\norm{\DFk'}^2,
			~~\text{with}~~
				\tauk
			\coloneqq
				\frac{
					\innprod{\Dxk}{\DFk'}
				}{
					\frac{\varepsilon}{2\gamk\rhok}\norm{\Dxk}^2
					+
					\frac{\rhok\gamk}{2\varepsilon}\norm{\DFk'}^2
				}
			\in[-1,1],
			\end{equation}
			while for the second one \eqref{eq:Youngcos_pos} with parameter \(\frac{\mu(1+\alpha\rhok*)}{\alpha\rhok}\) for some \(\mu>0\) gives
			\begin{equation}\label{eq:cosphi-bound-analysis}
				-\innprod{\DFk}{\DFk'}
			\leq
				[\cos\phi_k]_-
				\left(
					\tfrac{\mu(1+\alpha\rhok*)}{2\alpha\rhok}
					\norm{\DFk}^2
					+
					\tfrac{\alpha\rhok}{2\mu(1+\alpha\rhok*)}
					\norm{\DFk'}^2
				\right),
			\end{equation}
		\end{subequations}
		where \(\cos\phi_k\coloneqq\frac{\innprod{\DFk}{\DFk'}}{\norm{\DFk}\norm{\DFk'}}\) is the cosine of the angle between \(\DFk\) and \(\DFk'\).
		We may thus substitute the identity and the inequality of \eqref{eq:taukcosphik} into \cref{thm:main:innprods}, and after suitably rearranging the following is obtained (the detailed derivation is given in \cref{proof:thm:Lyapunov}).

		\begin{lemma}[Lyapunov-type descent]\label{thm:Lyapunov}%
			Fix \(\varepsilon,\mu,\lambda>0\) and \(\alpha>0\), and consider the iterates generated by \eqref{eq:alg-family}.
			For \(x\in\dom g\) and \(k\geq1\), let
			\[
				\Uk(x)
			\coloneqq
				\tfrac{1}{2}
				\norm{\xk-x}^2
				+
				\tfrac{1}{2}
				\norm{\Dxk}^2
				+
				\tfrac{\lambda(\gamk\rhok)^2}{2}
				\norm{\DFk'}^2
				+
				(1+\alpha\rhok)
				\gamk\Vk'(x),
			\]
			which is \eqref{eq:Uk} for a general \(\lambda\).
			Then, for any \(k\geq1\) it holds that
			\begin{gather}\label{eq:descent}
				\Uk*(x)
			\leq
				\Uk(x)
				-
				\alpha\gamk A_{k+1}
				\Vk'(x)
				-
				(\alpha\rhok\gamk)^2
				B_{k+1}
				\norm{\DFk'}^2
				-
				C_{k+1}
				\norm{\Dxk}^2,
			\\
				\label{eq:ABC}
				\text{where}
			\quad
				\left\{
					\begin{array}{@{}r @{} l@{}}
						A_{k+1}
					={} &
						\tfrac{1}{\alpha}+\rhok-\rhok*^2
					\\
						B_{k+1}
					={} &
						\tfrac{\lambda}{2\alpha^2}
						-
						\rhok*^2
						\Bigl(
							1
							+
							\tauk
							\tfrac{2-\alpha}{2\alpha\varepsilon}
							+
							\tfrac{[\cos\phi_k]_-}{\mu}
						\Bigr)
					\\
						C_{k+1}
					={} &
						\tfrac{1}{2}
						-
						(\gamk*\Lk)^2
						\Bigl(
							\tfrac{\lambda}{2}
							\rhok*^2
							+
							(1+\alpha\rhok*)^2
							\bigl(
								1
								+
								\mu
								[\cos\phi_k]_-
							\bigr)
						\Bigr)
					\\
					&
						+
						\Bigl(
							\alpha
							-
							1
							-
							\tauk\alpha\varepsilon
							\tfrac{2-\alpha}{2}
						\Bigr)
						\rhok*^2
					\end{array}
				\right.
			\end{gather}
			with \(\tauk\) and \(\cos\phi_k\) as in \eqref{eq:taukcosphik}, both bounded between \(-1\) and \(1\).
		\end{lemma}

		\subsection{A general convergence recipe}
			The preceding lemma isolates all stepsize-de\-pen\-dent requirements in the three coefficients \(A_{k+1}\), \(B_{k+1}\), and \(C_{k+1}\).
			We now derive the convergence consequences that follow from their nonnegativity, independently of the particular stepsize rule used to enforce it.
			This will reduce the subsequent analysis of \refadaFRB{} and \refadaFRBII{} to verifying that their respective stepsize updates satisfy these coefficient conditions, together with the additional stepsize conditions required for convergence.

			When the comparison point is a solution \(x^\star\in\zer T\), monotonicity of \(F\) and the optimality condition \(-F(x^\star)\in\partial g(x^\star)\) imply \(\Vk(x^\star)\geq0\).
			Consequently, \(\Uk(x^\star)\) is nonnegative, and \eqref{eq:descent} becomes a genuine Lyapunov descent inequality whenever \(A_{k+1},B_{k+1},C_{k+1}\geq0\).
			Note also that \(\Vk\) vanishing does not certify optimality of \(\xk\), and it need not be convex in the iterate, so its bounds do not pass to averages.
			For these reasons, ergodic estimates are measured by the restricted gap function \(\Gap_{x^\star,\delta}\) in \eqref{eq:Gap} \cite{facchinei2003finite,nesterov2007dual,malitsky2020golden}.
			We next list some basic properties for later reference. The proof is an easy adaptation of the one in \cite[Lem. 1]{nesterov2007dual} to account for arbitrary functions \(g\) as in \cref{ass:g}, as opposed to merely indicators.
			\begin{fact}[{\cite[Lem. 1]{nesterov2007dual}}]%
				Let \(\Gap_{x^\star,\delta}\) be as in \eqref{eq:Gap} with \(x^\star\in\zer T\) and \(\delta>0\).
				\begin{enumerate}
				\item
					\(\Gap_{x^\star,\delta}\) is convex and finite on \(\dom g\).
				\item \label{thm:Gap:solutions}%
					\(\Gap_{x^\star,\delta}(x)\geq0\) for all \(x\in\R^n\);
					for any \(x\in\ball{x^\star}{\delta}\), equality holds only if \(x\in\zer T\).
				\end{enumerate}
			\end{fact}

			\begin{theorem}[convergence recipe]\label{thm:recipe}%
				Consider the iterates generated by \eqref{eq:alg-family} for some \(\alpha>0\).
				Suppose that \(\gamk*=\rhok*\gamk\) is selected such that
				\[
					A_{k+1},B_{k+1},C_{k+1}\geq0
				\quad
					\forall k\in\N,
				\]
				where  \(A_{k+1},B_{k+1},C_{k+1}\) are as in \cref{thm:Lyapunov} for some \(\varepsilon,\mu,\lambda>0\).
				Then, for any \(x^\star\in\zer T\) and \(k\geq1\) the following hold:
				\begin{enumerate}
				\item \label{thm:descent}%
					\(0\leq\Uk*(x^\star)\leq\Uk(x^\star)\), and in particular \(\seq{\xk}\) is bounded.

				\item \label{thm:ergodic}%
						For any \(\delta\geq\sqrt{2\Uk_1(x^\star)}\), the rate
						\(
							\Gap_{x^\star,\delta}(\bar x^k)
						\leq
							\frac{
								\sup\set{
									\Uk_2(x)
								}[
									x\in\dom g\cap \ball{x^\star}{\delta}
								]
							}{
								\sum_{j=1}^k\gamma_{j+1}
							}
						\)
						holds for the ergodic sequence \(\seq{\bar x^k}\) given by \eqref{eq:barxk}.
						In particular, if \(\sum_{k\in\N}\gamk=\infty\) then all its cluster points solve \eqref{eq:P}.
				\item \label{thm:residual}%
					If \(C_{\rm min}\coloneqq\inf_{k\geq1}C_{k+1}>0\), then
					\(
						\min_{1\leq j\leq k}
						\norm{x^j-x^{j-1}}^2
					\leq
						\frac{\Uk_1(x^\star)}{C_{\rm min}k}
					\).
				\item \label{thm:limit}%
					If also \(\inf_k\gamk>0\) holds,
						\(\seq{\xk}\) converges to a solution and \(\displaystyle\min_{j\leq k}\dist^2(0,T(x^j))\sim O(\nicefrac{1}{k})\).
				\end{enumerate}
			\end{theorem}
			\begin{proof}
				In what follows, let \(x^\star\) be any solution to \eqref{eq:P}.
				\begin{itemize}[leftmargin=*, align=right]
				\item \ref{thm:descent}
					As discussed after \eqref{eq:Vk}, \(\Vk(x^\star)\geq0\) holds for any \(k\), and in particular it follows from its definition \eqref{eq:Uk} that \(\Uk(x^\star)\geq0\).
					Moreover, the inequality in \eqref{eq:descent} together with the nonnegativity of \(A_{k+1},B_{k+1},C_{k+1}\) implies that \(\Uk*(x^\star)\leq\Uk(x^\star)\), and since \(\frac{1}{2}\norm{\xk-x^\star}^2\leq\Uk(x^\star)\leq\cdots\leq\Uk_1(x^\star)\), boundedness of the sequence also follows.

				\item \ref{thm:ergodic}
					The condition \(A_{k+1}\geq0\) implies that \(\bar x^k\) is a convex combination of \(x^1,\dots,x^k\), each being the output of a proximal mapping of \(g\) and thus in \(\dom g\).
					Moreover, the above chain of inequalities together with \(\Uk_1(x^\star)\leq\frac{1}{2}\delta^2\) yields that \(\xk\in\ball{x^\star}{\delta}\) for any \(k\geq1\).
					Convexity of \(\dom g\) and of \(\ball{x^\star}{\delta}\) then implies that \(\bar x^k\in\dom g{}\cap\ball{x^\star}{\delta}\) too.
					For any \(x\in\dom g\), the inequality \eqref{eq:descent} yields that
					\begin{align*}
						\textstyle
						\sum_{j=2}^k
						\gamma_j
						\bigl(1+\alpha\rho_j-\alpha\rho_{j+1}^2\bigr)
						\Vk_{j-1}(x)
					\leq{} &
						\textstyle
						\sum_{j=2}^k
						\bigl[
							\Uk_j(x)
							-
							\Uk_{j+1}(x)
						\bigr]
					\\
					={} &
						\Uk_2(x)-\Uk*(x)
					\\
					\leq{} &
						\Uk_2(x)
						-
						(1+\alpha\rhok*)
						\gamk*\Vk(x).
					\numberthis\label{eq:Uk:sum:rate}
					\end{align*}
					For conciseness, let us denote by
					\begin{align}
					\nonumber
						D_k
					\coloneqq{} &
						\textstyle
						\gamk*(1+\alpha\rhok*)
						+
						\sum_{j=2}^k
						\alpha\gamma_j
						\bigl(
							\tfrac{1}{\alpha} + \rho_j - \rho_{j+1}^2
						\bigr)
					\\
					={} &
						\textstyle
						\gamk*(1+\alpha\rhok*)
						+
						\alpha
						\bigl(
							\gamma_2\rho_2
							-
							\gamk*\rhok*
						\bigr)
						+
						\sum_{j=2}^k
						\gamma_j
					\geq
						\sum_{j=1}^k
						\gamma_{j+1}
					\label{eq:DK}
					\end{align}
					the normalizing sum in the definition of \(\bar x^k\).
					Furthermore, let us define
					\[
						G(x,y)
					\coloneqq
						\innprod{F(x)}{y-x}
						+
						g(y)
						-
						g(x),
					\]
					so that from \eqref{eq:Vk} and the definition of \(\Gap_{x^\star,\delta}\) we have that
					\[
						\Vk(x)
					\geq
						G(x,x^k)
					~~
						\forall x\in\dom g
					\quad\text{and}\quad
						\Gap_{x^\star,\delta}(y)
					=
						\sup_{x\in\dom g\cap\ball{x^\star}{\delta}}
						G(x,y).
					\]
					Using
					the inequality above in \eqref{eq:Uk:sum:rate} together with the nonnegativity of the coefficients \(1+\alpha\rho_{k+1}\) and \(1+\alpha\rhok-\alpha\rhok*^2\), yields
					\[
						\textstyle
						\tfrac{\gamk*(1+\alpha\rhok*)}{D_k}
						G(x,x^k)
						+
						\sum_{j=2}^k
						\tfrac{\gamma_j(1+\alpha\rho_j-\alpha\rho_{j+1}^2)}{D_k}
						G(x,x^{j-1})
					\leq
						\tfrac{\Uk_2(x)}{D_k}
					\quad
						\forall x\in\dom g.
					\]
					Note that the left-hand side is a convex combination of \(G(x,x^1),\dots,G(x,\xk)\), whence convexity of \(y\mapsto G(x,y)\) and Jensen's inequality give that
					\[
						G(x,\bar x^k)
					\leq
						\tfrac{\Uk_2(x)}{D_k}
					\quad
						\forall x\in\dom g.
					\]
					Taking the supremum over \(x\in\dom g\) with \(\norm{x-x^\star}\leq\delta\) and using \eqref{eq:DK} yields the sought rate, where the denominator may be replaced by \(\sum_{j=1}^k\gamma_{j+1}\leq D_k\) since the supremum is at least \(\Uk_2(x^\star)\geq0\).
						If \(\sum_{k\geq2}\gamk=\infty\), then \(\Gap_{x^\star,\delta}(\bar x^k)\to0\), which by lower semicontinuity implies that \(\Gap_{x^\star,\delta}(\bar x)=0\) at each cluster point.
					Since any cluster point \(\bar x\) satisfies \(\norm{\bar x-x^\star}\leq\delta\), optimality of \(\bar x\) follows from \cref{thm:Gap:solutions}.
				\end{itemize}

				\noindent
				In what follows, suppose that \(C_{\rm min}\coloneqq\inf_{k\geq1}C_{k+1}>0\).
				\begin{itemize}
				\item \ref{thm:residual}
					Discarding the first two negative terms in \eqref{eq:descent}, which is legitimate as \(A_{k+1},B_{k+1}\geq0\) and \(\Vk(x^\star)\geq0\), and telescoping for \(j=1,\dots,k\) gives
					\[
						kC_{\rm min}
						\min_{1\leq j\leq k}
						\norm{\Dxk^j}^2
					\leq{}
						C_{\rm min}
						\sum_{j=1}^k
						\norm{\Dxk^j}^2
					\leq
						\sum_{j=1}^k
						C_{j+1}
						\norm{\Dxk^j}^2
					\leq
						\Uk_1(x^\star)
						-
						\Uk*(x^\star)
					\leq
						\Uk_1(x^\star),
					\]
					where the last inequality owes to \cref{thm:descent}.
				\end{itemize}

				\noindent
				To conclude, suppose further that \(\gammin\coloneqq\inf_{k\geq1}\gamk>0\).
				\begin{itemize}
				\item \ref{thm:limit}
					The condition \(A_{k+1}\geq0\) reads \(\rhok*\leq\sqrt{\nicefrac{1}{\alpha}+\rhok}\), implying that\footnote{%
						Without loss of generality we assumed that \(\rho_0\leq\rhomax\), e.g. when \(\rho_0=1\) as in the proposed algorithms.
						More generally, it is simply a matter of replacing this bound with \(\max\set{\rho_0,\rhomax}\).
					}
					\begin{equation}\label{eq:rhomax}
					\textstyle
						\rhok
					\leq
						\rhomax
					\coloneqq
						\tfrac{1}{2}
						\bigl(1+\sqrt{1+\frac{4}{\alpha}}\bigr)
					\quad
						\forall k\in\N.
					\end{equation}
					Moreover, \(F\) is Lipschitz continuous with some modulus \(L_{F,\V}\) on a compact set \(\V\) containing the iterates.
					Since \(F(\xk)-\uk=\DFk-\alpha\rhok\DFk'\), the subgradient \eqref{eq:tildeg} gives
					\begin{equation}\label{eq:dist-residual}
						\dist(0,T(\xk))
					\leq
						(\gammin^{-1}+L_{F,\V})\norm{\Dxk}
						+
						\alpha\rhomax L_{F,\V}\norm{\Dxk'}
					\leq
						M\bigl(\norm{\Dxk}+\norm{\Dxk'}\bigr),
					\end{equation}
					with \(M\coloneqq\max\set{\gammin^{-1}+L_{F,\V},\alpha\rhomax L_{F,\V}}\).
					Squaring, telescoping, and using the bound on \(\sum_k\norm{\Dxk}^2\) above yields the claimed rate on \(\dist^2(0,T(\xk))\).
					Since the graph of \(T\) is closed, it follows that any cluster point of \(\seq{\xk}\) is a solution.
					As such, the proof is concluded once we show that the sequence converges.
					Denoting \(D\coloneqq\max\{0,\alpha-1+\alpha\varepsilon|2-\alpha|/2\}\), \eqref{eq:ABC} yields
					\[
						0\leq C_{k+1}\leq\tfrac12+D\rhomax^2-(\gamk*\Lk)^2,
					\]
					hence \(\gamk*\Lk\leq\sqrt{\frac12+D\rhomax^2}\) for all \(k\in\N\).
					This bound combined with \(\rhok\leq\rhomax\) yields that
					\[
						\gamk\norm{\uk'-F(\xk')}
					\leq
						\gamk\Lk'\norm{\Dxk'}
						+\alpha\rhok\rhok'\gamk'\Lk''\norm{\Dxk''}
					\to
						0.
					\]
					Multiplying the bound in \cref{thm:descent-star} at \(k-1\) by \(\rhok\), boundedness of the iterates and \(\Dxk\to0\) give \(\gamk\Vk'(x^\star)\to0\).
					Let \(x^\infty\) be a cluster point.
					Since \(x^\infty\) is a solution, the above holds with \(x^\star=x^\infty\). Thus, along a subsequence converging to \(x^\infty\), all terms of \(\Uk(x^\infty)\) tend to zero.
					Monotonicity of \(\Uk(x^\infty)\) then gives \(\Uk(x^\infty)\to0\), and \(\frac12\norm{\xk-x^\infty}^2\leq\Uk(x^\infty)\) proves sequence convergence.
				\qedhere
				\end{itemize}
			\end{proof}

		\subsection{Proofs of the main results}
			In this subsection we detail the proofs of the main convergence results of the proposed algorithms.
			In all cases, we will first identify \(\varepsilon,\lambda,\mu>0\) such that the coefficients \(A_{k+1},B_{k+1},C_{k+1}\) in \eqref{eq:ABC} are nonnegative, so as to invoke the general convergence recipe of \cref{thm:recipe}.
			To this end, we recall that \(\abs{\tauk}\leq1\) and \([\cos\phi_k]_-\leq1\).
			Therefore, for any \(\varepsilon,\lambda,\mu>0\) one has that
			\begin{equation}\label{eq:ABCwc}
				\left\{
					\begin{array}{@{}r @{} l@{}}
						A_{k+1}
					={} &
						\tfrac{1}{\alpha}+\rhok-\rhok*^2
					\\
						B_{k+1}
					\geq{} &
						\tfrac{\lambda}{2\alpha^2}
						-
						\rhok*^2
						\Bigl(
							1
							+
							\tfrac{\abs{2-\alpha}}{2\alpha\varepsilon}
							+
							\tfrac{1}{\mu}
						\Bigr)
					\\
						C_{k+1}
					\geq{} &
						\tfrac{1}{2}
						-
						(\gamk*\Lk)^2
						\Bigl(
							\tfrac{\lambda}{2}
							\rhok*^2
							+
							(1+\mu)
							(1+\alpha\rhok*)^2
						\Bigr)
						+
						\Bigl(
							\alpha
							-
							1
							-
							\alpha\varepsilon
							\tfrac{\abs{2-\alpha}}{2}
						\Bigr)
						\rhok*^2.
					\end{array}
				\right.
			\end{equation}

		\subsubsection[Proof of Theorem \ref*{thm:FRBa}]{Proof of \cref{thm:FRBa}}\label{proof:thm:FRBa}%
			Employing a constant stepsize \(\gamk\equiv\gamma>0\) implies that \(\rhok=\frac{\gamk}{\gamk'}\equiv1\).
			Combined with the bound \(\Lk\leq L_F\), it follows from \eqref{eq:ABCwc} that
			\[
				\left\{
					\begin{array}{@{}r @{} l@{}}
						A_{k+1}
					={} &
						\tfrac{1}{\alpha}
					\\
						B_{k+1}
					\geq{} &
						\tfrac{\lambda}{2\alpha^2}
						-
						1
						-
						\tfrac{\abs{2-\alpha}}{2\alpha\varepsilon}
						-
						\tfrac{1}{\mu}
					\\
						C_{k+1}
					\geq{} &
						\alpha
						-
						\tfrac{1}{2}
						-
						\alpha\varepsilon
						\tfrac{\abs{2-\alpha}}{2}
						-
						(\gamma L_F)^2
						\bigl(
							\tfrac{\lambda}{2}
							+
							(1+\mu)
							(1+\alpha)^2
						\bigr).
					\end{array}
				\right.
			\]
			The particular choice
			\[
				\lambda
			=
				2\alpha^2
				\Bigl(
					1
					+
					\tfrac{\abs{2-\alpha}}{2\alpha\varepsilon}
					+
					\tfrac{1}{\mu}
				\Bigr),
			\quad
				\mu
			=
				\tfrac\alpha{1+\alpha},
			\quad\text{and}\quad
				\varepsilon=c(\alpha)
			\]
			results in
			\[
				A_{k+1}
			=
				\tfrac{1}{\alpha},
			\quad
				B_{k+1}
			\geq
				0,
			\quad\text{and}\quad
				C_{k+1}
			\geq
				C
				\bigl(c(\alpha)^2-(\gamma L_F)^2\bigr),
			\]
			where
			\(
				C
			=
				\alpha\abs{2-\alpha}
				\tfrac{1}{2c(\alpha)}
				+
				(1+2\alpha)^2
			\)
			as in the statement.
			Thus, \(\inf_{k\in\N}C_{k+1}>0\) as long as \(\gamma L_F<c(\alpha)\), and the parameters \(\varepsilon,\lambda,\mu\) were chosen to maximize this admissible range of \(\gamma\).
			All the claims of \cref{thm:FRBa} then follow from \cref{thm:recipe}.

		\subsubsection[Proof of Theorem \ref*{thm:adaFRB}]{Proof of \cref{thm:adaFRB}}\label{proof:thm:adaFRB}%
			The proof proceeds in three steps: in the first two we verify that the coefficients \(A_{k+1},B_{k+1},C_{k+1}\) satisfy the conditions of \cref{thm:recipe} for both \refadaFRB{} and \refadaFRBII{}, and in the last part we derive the convergence claims for the two algorithms simultaneously.

			\paragraph{Coefficients for \refadaFRB}
				Throughout this proof \(\alpha\in[1,2]\) is fixed, so that the absolute values in \eqref{eq:ABCwc} can be dropped, and \(c=\frac{1}{7-\alpha}\) denotes the safety coefficient set at the initialization of \refadaFRB.
				By \cref{thm:recipe}, it suffices to show that \(A_{k+1},B_{k+1},C_{k+1}\geq0\) for every \(k\in\N\), with \(C_{k+1}\) bounded away from zero.

				The first term in the minimum at \cref{state:adaFRB:rhok*} enforces \(\rhok*\leq\sqrt{\frac{1}{\alpha}+\rhok}\), which is precisely \(A_{k+1}\geq0\).
				It also yields \eqref{eq:rhomax}.
				Combined with the second term, which enforces \(\rhok*\leq\frac{2}{3}+\frac{2\alpha}{5}\), the effective bound on the stepsize ratio is thus
				\begin{equation}\label{eq:adaFRB:beta}
					\rhok*
				\leq
					\beta
				\coloneqq
					\min\set{
						\tfrac{2}{3}+\tfrac{2\alpha}{5}
					,\,
						\rhomax
					}
				\qquad
					\forall k\in\N,
				\end{equation}
				while the third one enforces \(\gamk*\Lk\leq c\).
				It remains to choose the parameters \(\varepsilon\), \(\mu\) and \(\lambda\) of \cref{thm:Lyapunov}.
				They are dictated by the requirements \(B_{k+1}\geq0\) and \(C_{k+1}\geq0\) in the worst case \(\rhok*=\beta\) and \(\tauk=[\cos\phi_k]_-=1\).
				We take \(\lambda\) as the smallest value for which \(B_{k+1}\geq0\) in this case.
				The requirement \(C_{k+1}\geq0\) then restricts \(\gamk*\Lk\), and \(\mu\) is chosen to make this restriction as mild as possible.
				The parameter \(\varepsilon\) enters both \(B_{k+1}\) and \(C_{k+1}\),
				and optimizing it leads to an implicit equation.

				We instead fix \(\varepsilon=\nicefrac{1}{5}\), which keeps the resulting condition on \(\gamk*\Lk\) explicit and only marginally more restrictive than the optimal one for every \(\alpha\in[1,2]\).
				Specifically, we take
				\begin{equation}\label{eq:adaFRB:coeff}
					\mu
				=
					\tfrac{\alpha\beta^2}{1+\alpha\beta},
				\quad
					\varepsilon = \tfrac{1}{5},
				\quad\text{and}\quad
					\lambda
				=
					2(\alpha\beta)^2
					\Bigl(
						1
						+
						\tfrac{5(2-\alpha)}{2\alpha}
						+
						\tfrac{1}{\mu}
					\Bigr).
				\end{equation}
				With these values, the bound \(\rhok*\leq\beta\) applied to \eqref{eq:ABCwc} gives
				\[
					B_{k+1}
				\geq
					\tfrac{\lambda}{2\alpha^2}
					-
					\beta^2
					\Bigl(
						1
						+
						\tfrac{2-\alpha}{2\alpha\varepsilon}
						+
						\tfrac{1}{\mu}
					\Bigr)
				=
					0,
				\]
				the equality being the definition of \(\lambda\) in \eqref{eq:adaFRB:coeff}.
				As to \(C_{k+1}\), its lower bound in \eqref{eq:ABCwc} is a quadratic function of \(\rhok*\) whose derivative
				\[
					\rhok*
					\bigl(
						2\alpha-2-\alpha\varepsilon(2-\alpha)
						-
						c^2\lambda
						-
						2c^2\alpha^2(1+\mu)
					\bigr)
					-
					2c^2\alpha(1+\mu)
				\]
				is affine in \(\rhok*\) and equals \(-2c^2\alpha(1+\mu)<0\) at the origin, so its sign on \([0,\beta]\) is decided at \(\rhok*=\beta\), where it is negative for every \(\alpha\in[1,2]\), as a direct computation shows.
				The minimum of \(C_{k+1}\) over the admissible range is therefore attained at \(\rhok*=\beta\).
				Together with \(\gamk*\Lk\leq c\) this yields
				\[
					C_{k+1}
				\geq
					\tfrac{1}{2}
					\Bigl(
						1+\beta^2\bigl(2\alpha-2-\alpha\varepsilon(2-\alpha)\bigr)
					\Bigr)
					-
					c^2
					\Bigl(
						\tfrac{\lambda}{2}
						\beta^2
						+
						(1+\mu)
						(1+\alpha\beta)^2
					\Bigr),
				\]
				where the terms not involving \(c\) have been collected.
				Substituting \(\varepsilon\), \(\lambda\) and \(\mu\) from \eqref{eq:adaFRB:coeff} and using the identity \(\alpha^2\beta^4(1+\nicefrac1\mu)+(1+\mu)(1+\alpha\beta)^2=(1+\alpha\beta+\alpha\beta^2)^2\), this reads
				\begin{equation}\label{eq:adaFRB:Cmin}
					C_{k+1}
				\geq
					\tfrac{1}{2}
					\Bigl(
						1+\beta^2\bigl(2\alpha-2-\tfrac{\alpha(2-\alpha)}{5}\bigr)
					\Bigr)
					-
					c^2
					\Bigl(
						(1+\alpha\beta+\alpha\beta^2)^2
						+
						\tfrac{5}{2}\alpha(2-\alpha)\beta^4
					\Bigr)
				\eqqcolon
					C_{\min}.
				\end{equation}
				With \(c=\frac{1}{7-\alpha}\) and \(\beta\) as in \eqref{eq:adaFRB:beta}, \(C_{\min}\) depends on \(\alpha\) alone, and is strictly positive for every \(\alpha\in[1,2]\), as a direct computation shows.

			\paragraph{Coefficients for \refadaFRBII}
				Throughout this part \(\alpha=2\) is fixed, resulting in the terms carrying the factor \(2-\alpha\) in \eqref{eq:ABC} vanishing.
				We abbreviate \(t_k\coloneqq[\cos\phi_k]_-\in[0,1]\).
				We first verify \(A_{k+1},B_{k+1}\geq0\) and \(C_{k+1}\geq C_{\rm min}\) for every \(k\in\N\), where \(C_{\rm min}\coloneqq\nicefrac{6}{100}-c^2>0\) since \(c<\nicefrac{\sqrt6}{10}\).

				The first term in the minimum at \cref{state:adaFRB+2:gamk*} enforces \(\rhok*\leq\sqrt{\nicefrac{1}{2}+\rhok}\), which is precisely \(A_{k+1}\geq0\).
				The second term enforces \(\rhok*\leq\betk\), while the third one enforces \(\gamk*\Lk\leq c\).

				We take the parameters of \cref{thm:Lyapunov} as
				\[
					\mu=\tfrac{2}{3}
				\quad\text{and}\quad
					\lambda=8\bigl(1+\tfrac{1}{\mu}\bigr)=20.
				\]
				For fixed \(\mu\), this choice of \(\lambda\) makes \(B_{k+1}\geq0\) at \(\rhok*=1\) for every \(t_k\in[0,1]\). Choosing \(\mu=2/3\) maximizes the resulting admissible cap.
				With these values, \(B_{k+1}\geq0\) in \eqref{eq:ABC} reads
				\begin{equation}\label{eq:adaFRB+2:B}
					\rhok*^2
				\leq
					\tfrac{1+\frac{1}{\mu}}{1+\frac{t_k}{\mu}}
				=
					\tfrac{5}{2+3t_k}
				=
					\betk^2,
				\end{equation}
				which is enforced by the second term at \cref{state:adaFRB+2:gamk*}.
				As to \(C_{k+1}\), using \(\gamk*\Lk\leq c\), \eqref{eq:ABC} gives
				\begin{equation}\label{eq:adaFRB+2:C}
					C_{k+1}
				\geq
					\tfrac{1}{2}
					+
					\rhok*^2
					-
					c^2\Psi_k,
				\qquad
					\Psi_k
				\coloneqq
					10\rhok*^2
					+
					(1+\mu t_k)(1+2\rhok*)^2
				\geq
					1.
				\end{equation}
				Write \(r=\rhok*\) and \(t=t_k\).
				Since \(\mu=\frac23\), one has \(\frac{6}{100}\Psi_k=\frac35r^2+\frac{3+2t}{50}(1+2r)^2\).
				Substituting \(c^2=\frac{6}{100}-C_{\rm min}\) in \eqref{eq:adaFRB+2:C} and expanding \((1+2r)^2\) then gives
				\begin{equation}\label{eq:adaFRB+2:qt}
					C_{k+1}
				\geq
					\tfrac{1}{25}q_t(r)+C_{\rm min}\Psi_k,
				\qquad
					q_t(r)\coloneqq4(1-t)r^2-2(3+2t)r+11-t.
				\end{equation}
				Since \(\Psi_k\geq1\), it remains to show \(q_t(r)\geq0\).
				Rewriting \(q_t(r)=4r^2-6r+11-t(2r+1)^2\) shows that \(q_t(r)\) is decreasing in \(t\), so it suffices to check it at the largest admissible value of \(t\).
				For \(r\leq1\) this is \(t=1\), and \(q_1(r)=10(1-r)\geq0\).
				For \(r\geq1\), rewriting \eqref{eq:adaFRB+2:B} as \(2+3t\leq\frac{5}{r^2}\) gives \(t\leq\bar t\coloneqq\frac{5-2r^2}{3r^2}\).
				Since \(1-\bar t=\frac{5(r^2-1)}{3r^2}\), \(3+2\bar t=\frac{5(r^2+2)}{3r^2}\) and \(11-\bar t=\frac{5(7r^2-1)}{3r^2}\), we obtain
				\[
					q_t(r)
				\geq
					q_{\bar t}(r)
				=
					\tfrac{20r^2(r^2-1)-10r(r^2+2)+5(7r^2-1)}{3r^2}
				=
					\frac{5(r-1)(4r^3+2r^2+5r+1)}{3r^2}
				\geq
					0.
				\]

			\paragraph{Conclusion, common to both schemes}
				Having verified that \(A_{k+1},B_{k+1},C_{k+1}\geq0\) for every \(k\in\N\), \cref{thm:recipe} applies.
				In particular, \cref{thm:descent} gives the Lyapunov descent and the boundedness of \(\seq{\xk}\), so that \cref{ass:f} provides a compact and convex set \(\V\) containing all the iterates on which \(F\) is Lipschitz continuous with some modulus \(L_{F,\V}<\infty\).
				Then \(\Lk\leq L_{F,\V}\), hence \(\nicefrac{c}{\Lk}\geq\nicefrac{c}{L_{F,\V}}\eqqcolon s_{\rm min}>0\), for every \(k\in\N\).
				In either scheme the second term of the stepsize update is at least \(\gamk\), since \(\frac{2}{3}+\frac{2\alpha}{5}>1\) at \cref{state:adaFRB:rhok*} and \(\betk\geq1\) at \cref{state:adaFRB+2:gamk*}, whence
				\[
					\gamk*
				\geq
					\min\set{
						\gamk
					,\,
						\gamk\sqrt{\tfrac{1}{\alpha}+\rhok}
					,\,
						s_{\rm min}
					}
				\qquad
					\forall k\in\N,
				\]
				and \cref{thm:gammin} yields assertion \ref{thm:adaFRB:gammin}.
				In turn, \(\inf_{k\in\N}\gamk>0\), so that \(\sum_{k=1}^{K+1}\gamk\) grows at least linearly in \(K\).
				Combined with the fact that \(C_{k+1}\geq C_{\rm min}>0\) for every \(k\in\N\), \cref{thm:limit} shows that \(\seq{\xk}\) converges to a solution, while \cref{thm:residual} gives the rate \(O(1/K)\) for \(\min_{k\leq K}\norm{\xk-\xk'}^2\) and \(\min_{k\leq K}\dist^2(0,T(\xk))\), and \cref{thm:ergodic} for the restricted gap function at \(\bar x^K\).

	\section{Simulations}\label{sec:simulations}
		We test the proposed methods on several standard monotone inclusion problems, using the open source code of \cite[\S6]{upadhyaya2026lyapunov} in Python 3.13.
		Unless otherwise specified, convergence is measured through a computable residual.
		Each algorithm retrieves
		\begin{equation}\label{eq:proxmapping}
			\hat x^k=\prox_{\gamk* g}\bigl(y^k-\gamk* d_k\bigr)
		\end{equation}
		from a point \(y^k\) along a forward direction \(d_k\), and the optimality condition of the proximal mapping gives
		\begin{equation}\label{eq:rk}
			r^{k+1}
		\coloneqq
			\tfrac{y^k- \hat x^k}{\gamk*} - d_k + F(\hat x^k)
		\in
			F(\hat x^k)+\partial g(\hat x^k)
		=
			T(\hat x^k).
		\end{equation}
		Our plots report \(\norm{r^{k+1}}\).
		The quantities \(y^k\), \(d_k\) and \(\gamk*\) are determined by the proximal update rule of each algorithm, and when \(g\equiv0\) the residual reduces to \(r^{k+1}=F(\hat x^k)\).
		Although \(d_k\) depends on \(\alpha\), by \eqref{eq:rk} the norm of \(r^{k+1}\) measures optimality in the same way for all tested methods.
		Runs stop at \(\norm{r^{k+1}}\leq10^{-10}\) or at a maximum iteration count.
		The horizontal axis counts evaluations of \(F\), since \(\prox_g\) has a closed form throughout and the methods differ in evaluations per iteration.
		Stepsizes are shown over the first \(200\) iterations, scaled by \(L_F\) when \(F\) is globally Lipschitz continuous, with the baseline \(\gamk L_F=1\) in gray.

		The baselines are the extragradient method (\EG) \cite{korpelevich1976extragradient}, forward-backward-forward splitting (\FBF) \cite{tseng2000modified}, forward-reflected-backward splitting (\frb) \cite{malitsky2020forward}, and the adaptive golden ratio method (\AGRAAL) \cite{alacaoglu2023beyond,malitsky2020golden}.
		Full descriptions of all methods are given in \cref{sec:sim_methods}.
		Each problem uses only the compatible methods, and fixed-stepsize methods run at \(0.9\) of their largest theoretical stepsize unless stated otherwise.
		We evaluate \refadaFRB{} with \(\alpha=1\) and \(\alpha=2\), denoted \adaFRB_1 and \adaFRB_2, together with \refadaFRBII.
		By \FRB_2 we denote the constant-stepsize scheme \eqref{eq:FRBa} of \cref{sec:alg_constant} with \(\alpha=2\), its \(\alpha=1\) instance being \frb, for which the tailored analysis of \cite{malitsky2020forward} permits a larger stepsize.
		Dashed lines mark nonadaptive methods, which require \(F\) to be globally Lipschitz, and solid lines with markers adaptive ones.
		Overlapping nonadaptive curves are plotted once, as noted in the legend and caption.

		\subsection{Cournot--Nash equilibrium problems}\label{sec:cnep}
			We begin with two classes of Cournot--Nash equilibrium problems, a standard family of benchmarks for monotone variational inequalities.
			In a noncooperative game with \(n\) players, each selecting a strategy \(x_i\in \mathcal X_i\) and incurring a loss \(\phi_i\), a pure-strategy Nash equilibrium is a point \(x\in \mathcal X\coloneqq\prod_{i=1}^n \mathcal X_i\) such that
			\[
				x_i
			\in
				\argmin_{y\in \mathcal X_i}\phi_i(y;x_{\setminus i}),
			\qquad
				i=1,\dots,n,
			\]
			where \(x_{\setminus i}\) denotes the strategies of all players except player \(i\).
			Under standard monotonicity assumptions this is an instance of \eqref{eq:P} with \(F(x)=\bigl(\nabla_{x_i}\phi_i(x)\bigr)_{i=1}^n\) and \(g=\delta_{\mathcal X}\).
			We consider a linear Cournot model, as well as a nonlinear variant.

			\subsubsection{Linear Cournot--Nash model}\label{sec:CNlinear}
				We first consider the classical Cournot--Nash model for an oligopolistic market with differentiated products.
				Each producer \(i\) chooses a production level \(x_i\in[0,T_i]\), so that \(\mathcal X=\prod_{i=1}^n[0,T_i]\), and the production cost, inverse demand and loss functions are
				\[
					c_i(x_i)=a_ix_i^2+b_ix_i,
				\quad
					p_i(x)=m_i-d_i\sum_{j=1}^n x_j,
				\quad
					\phi_i(x)=c_i(x_i)-x_ip_i(x).
				\]
				The associated operator takes the affine form
				\[
					F(x)=Ax+q,
				\qquad
					A=2\diag(a)+\diag(d)+d\mathbf 1^\top,
				\qquad
					q=b-m.
				\]
				Hence the equilibrium problem reduces to a linear monotone variational inequality with Lipschitz constant \(L_F=\norm{A}\), and we report results for \(n=10\) and \(n=100\).
				The problem data are generated exactly as in \cite[\S6]{upadhyaya2026lyapunov}.

			\subsubsection{Nonlinear Cournot--Nash model}\label{sec:CNnonlinear}
				We next consider a nonlinear extension of the previous model following \cite{baghbadorani2026hybrid}.
				Each producer chooses a production level \(x_i\in\R_+\), and the equilibrium is characterized by the variational inequality
				\[
					\innprod{F(x^\star)}{x-x^\star}
				\geq
					0,
				\qquad
					\forall x\in\R_+^n,
				\]
				where, writing \(Q\coloneqq\sum_{j=1}^n x_j\) for the aggregate production,
				\[
					F_i(x)
				=
					f_i'(x_i)-p(Q)-x_ip'(Q),
				\qquad
					p(Q)
				=
					5000^{1/\eta}Q^{-1/\eta},
				\]
				and \(f_i(x_i)=c_ix_i+\tfrac{\beta_i}{\beta_i+1}\kappa_i^{1/\beta_i}x_i^{(\beta_i+1)/\beta_i}\).
				In contrast to the linear model, both the demand and the cost functions depend nonlinearly on the production levels, yielding a nonlinear monotone operator \(F\) whose smoothness is governed by \(\eta\), \(\beta_i\) and \(\kappa_i\).

			\begin{figure}[!htbp]
				\centering
				\includetikz[width=\linewidth]{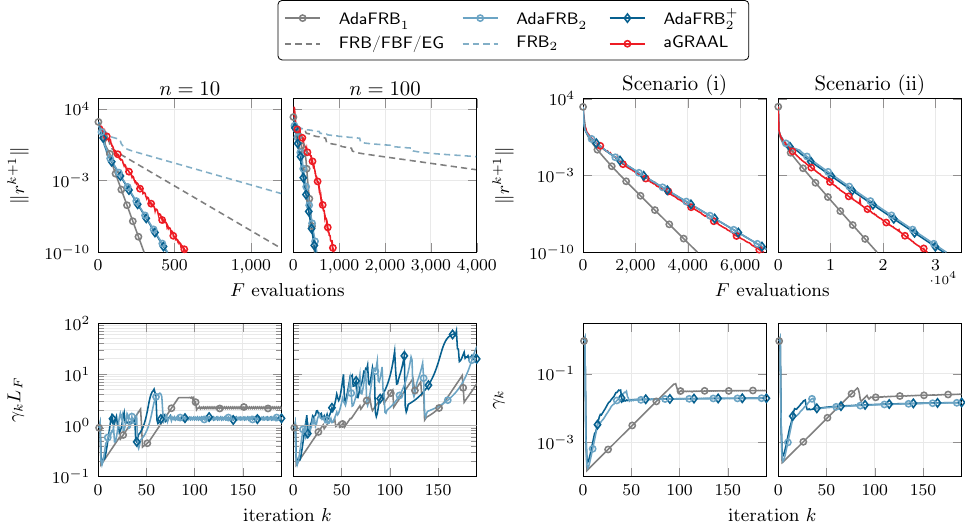}%
				\caption{%
					Cournot--Nash equilibrium problems of \S\ref{sec:CNlinear} (linear, left) and \S\ref{sec:CNnonlinear} (nonlinear, right), with overlapping plots of \frb, \EG{} and \FBF{} merged.
					The bottom rows report the stepsize magnitude over the first \(200\) iterations, scaled by \(L_F\) in the linear case, where the gray horizontal line marks \(\gamk L_F=1\).
					Only adaptive methods appear in the nonlinear case, where \(F\) is not globally Lipschitz continuous.
				}
				\label{fig:cnep}
			\end{figure}

			We set \(n=1000\) and draw the parameters randomly in two scenarios, namely \(\eta=1.1\) with \(\beta_i\sim\mathcal U(0.5,2)\), and \(\eta=1.5\) with \(\beta_i\sim\mathcal U(0.3,4)\), in both cases with \(c_i\sim\mathcal U(1,100)\) and \(\kappa_i\sim\mathcal U(0.5,5)\).
			\Cref{fig:cnep} reports the simulation results, all methods being initialized from the same randomly generated feasible point.
			In the linear setting the proposed adaptive methods require fewer evaluations of \(F\) than the other methods, and in the nonlinear setting the \(\alpha=1\) variant outperforms \AGRAAL, a comparable adaptive method.
			The nonlinear model is also the only instance considered here in which \(F\) fails to be globally Lipschitz continuous, so that the constant-stepsize methods are not applicable.

		\subsection{Bilinear zero-sum game with simplex constraints}\label{sec:bilinear_game}
			Consider the bilinear zero-sum game with simplex constraints given by
			\begin{equation}\label{eq:bilinear_game}
				\min_{x \in \Delta_{N}} \max_{y \in \Delta_{N}} x^\top\! A y,
			\end{equation}
			where \(A\in\R^{N\times N}\) is the payoff matrix and \(\Delta_{N}=\set{w\in\R^{N}_+}[w^\top\mathbf 1=1]\) is the probability simplex in \(\R^{N}\).
			Finding a saddle point of \eqref{eq:bilinear_game} amounts to solving \eqref{eq:P} in \(\R^{2N}\) with \(g=\delta_{\Delta_{N}\times\Delta_{N}}\) and
			\(
				F(z)
			=
					\big(Ay,
					-A^\top\! x
					\big)
			\)
			with \(z=(x,y)\) and \(L_F=\norm*{A}\).
					The payoff matrix is \(A=U\operatorname{diag}(s_1,\dots,s_N)V^\top\), with \(U\) and \(V\) Haar-distributed orthogonal matrices and \(s_j=10^{-3(j-1)/(N-1)}\), so that \(A\) has condition number \(10^3\). We take \(N=500\) and \(N=1000\).
					\Cref{fig:bilinear} reports the simulation results.
					Here \(F\) is globally Lipschitz continuous, and the proposed schemes reach lower residuals than the other tested methods without using \(L_F\).

			\begin{figure}[!htbp]
				\centering
				\includetikz[width=0.6\linewidth]{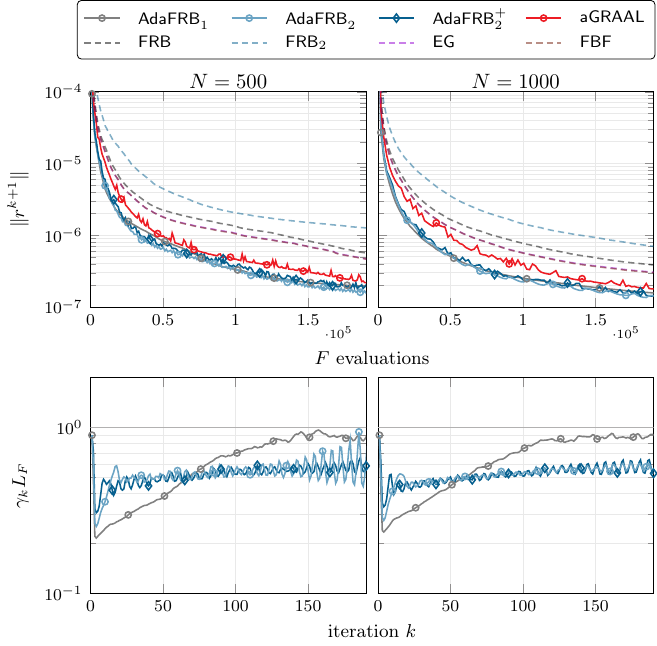}%
				\caption{%
					Bilinear zero-sum game \eqref{eq:bilinear_game} with simplex constraints of \S\ref{sec:bilinear_game} (top), and stepsize magnitude over the first \(200\) iterations scaled by \(L_F\) (bottom), where the horizontal gray line marks \(\gamk L_F=1\).%
				}%
				\label{fig:bilinear}
			\end{figure}

		\subsection{Sparse logistic regression}\label{sec:logreg}
			We finally consider the \(\ell_1\) regularized sparse logistic regression problem
			\begin{equation}\label{eq:logreg}
				\min_{x\in\R^n} \varphi(x)\coloneqq
				\sum_{i=1}^{N}
				\log\bigl(1+\exp(-b_i\innprod{a_i}{x})\bigr)
				+
				\lambda\norm{x}_1,
			\end{equation}
			where \(a_i\in\R^n\) and \(b_i\in\{-1,1\}\) for \(i=1,\ldots,N\), and \(\lambda > 0\) is the regularization parameter.
			Its optimality condition is an instance of \eqref{eq:P} with \(g=\lambda\norm{{}\cdot{}}_1\) and
			\(
				F(x)
			=
				\textstyle\sum_{i=1}^{N} -b_ia_i\,\sigma(-b_i\innprod{a_i}{x})
			=
				K^\top\sigma(Kx)
			\),
			\(
				K
			=
				\big(-b_1a_1,\ldots,-b_Na_N\big)^\top
			\in
				\R^{N\times n},
			\)
			where \(\sigma:\R^{N}\to\R^{N}\) applies \(u\mapsto\tfrac{\exp(u)}{1+\exp(u)}\) componentwise.
			The operator \(F\) is Lipschitz continuous with constant \(L_F=\tfrac14\norm{K}^2\).
			We use the datasets \texttt{spambase} and \texttt{a9a} from \cite{chang2011libsvm}, in each case with \(\lambda=\tfrac{1}{{N}}\) and \(\lambda=\tfrac{100}{{N}}\).
			We compute the optimal value \(\varphi_\star\) by running FISTA for \(2\cdot10^6\) iterations, and \cref{fig:logreg} reports the optimality gap \(\varphi(\xk*)-\varphi_\star\).
			At \(\alpha=2\), \refadaFRB_2 and \refadaFRBII{} behave almost identically.
			\begin{figure}[!htbp]
				\centering
				\includetikz[width=\linewidth]{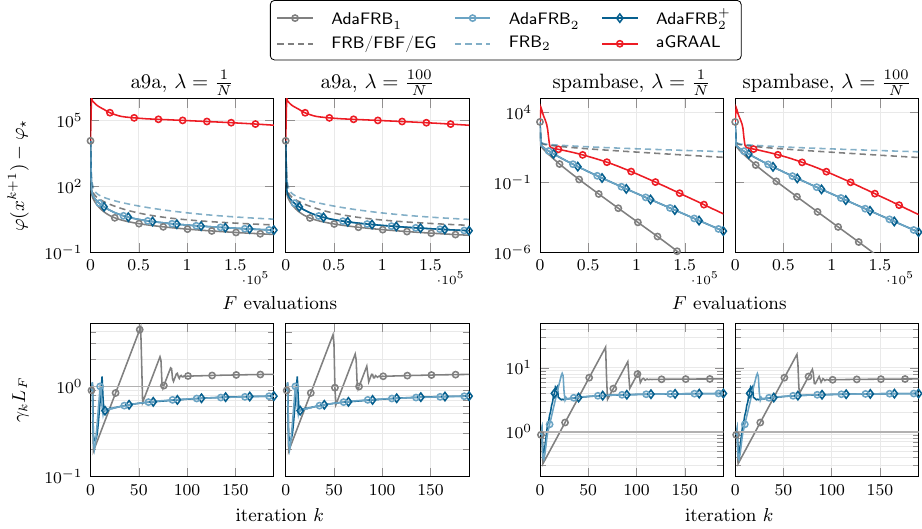}%
				\caption{%
					Sparse logistic regression \eqref{eq:logreg} of \S\ref{sec:logreg}, with overlapping plots of \frb, \EG{} and \FBF{} merged.
					The top row reports the optimality gap \(\varphi(\xk*)-\varphi_\star\).
					The bottom row reports the stepsize magnitude over the first \(200\) iterations scaled by \(L_F\), where the horizontal gray line marks \(\gamk L_F=1\).
				}
				\label{fig:logreg}
			\end{figure}

	\appendix
		\phantomsection
		\appendixnotitle
		\subsection{Auxiliary lemmas}
			\begin{lemma}[FNE-like lemma]\label{thm:FNE}%
				Given \(\xk',\uk,\uk*\in \R^n\), let \(\xk, \xk*\) be generated via \eqref{eq:xk*}.
				Then, denoting
				\[
					\wk*
				\coloneqq
					-\Dxk-\gamk(\uk-\uk*),
				\]
				and recalling \(\rhok* = \frac{\gamk*}{\gamk}\),
				for any \(k\geq1\) it holds that
				\[
					\norm{\Dxk*}^2
				\leq
					-\rhok*
					\innprod{
						\Dxk*
					}{
						\wk*
					}
				\leq
					\rhok*^2
					\norm{\wk*}^2.
				\]
			\end{lemma}
			\begin{proof}
				Follows from firm nonexpansiveness of \(\prox_{\gamk*g}\), after observing that
				\[
					\xk
				=
					\prox_{\gamk*g}\bigl(\xk+\gamk*\nabla*g(\xk)\bigr)
				=
					\prox_{\gamk*g}\bigl(\xk - \rhok*(\Dxk+\gamk\uk)\bigr),
				\]
				hence that
				\begin{align*}
					\norm{\xk*-\xk}^2
				\leq{} &
					\innprod{
						\xk*-\xk
					}{
						\xk-\gamk*\uk*
						-
						\bigl[\xk-\rhok*(\Dxk+\gamk\uk)\bigr]
					}
				\\
				={} &
					\rhok*
					\innprod{
						\xk*-\xk
					}{
						\Dxk
						+
						\gamk(\uk-\uk*)
					}.
				\end{align*}
				An application of the Cauchy-Schwarz inequality concludes the proof.
			\end{proof}

			\begin{lemma}\label{thm:descent-star}%
				For every \(k\geq1\) and \(x\in\dom g\),
				\[
					\gamk\Vk(x)
				\leq
					\tfrac{1}{2}
					\norm{\xk'-x}^2
					-
					\tfrac{1}{2}
					\norm{\xk-x}^2
					-
					\tfrac{1}{2}
					\norm{\xk'-\xk}^2
					+
					\gamk
					\innprod{\uk-F(\xk)}{x-\xk}.
				\]
			\end{lemma}
			\begin{proof}
				Since \(-\frac{1}{\gamk}\Dxk-\uk\in\partial g(\xk)\) by \eqref{eq:tildeg}, the subgradient inequality and \eqref{eq:Vk} give
				\[
					\gamk\Vk(x)
				\leq
					\innprod{\Dxk}{x-\xk}
					+
					\gamk\innprod{\uk-F(\xk)}{x-\xk}.
				\]
				Expanding \(\innprod{\Dxk}{x-\xk}=\frac12\norm{\xk'-x}^2-\frac12\norm{\xk-x}^2-\frac12\norm{\Dxk}^2\) yields the claim.
			\end{proof}

			\begin{lemma}\label{thm:gammin}%
				Starting from \(\gamma_0>0\), suppose that a sequence \(\seq{\gamk}\) satisfies
				\[
					\gamk*
				\geq
					\min\set*{\gamk\sqrt{\nicefrac{1}{\alpha}+\rhok}, \gamk, s_k}
				\]
				for some \(\alpha\in[1,2]\), where \(\rho_0\geq1\) and \(\rhok*=\frac{\gamk*}{\gamk}\) for \(k\geq0\), and that there exists \(s_{\rm min}>0\) such that \(s_k\geq s_{\rm min}\) for all \(k\).
				Then,
				\[
					\gamk
				\geq
					\tfrac{1}{\sqrt{\alpha}}
					\min\set{\gamma_0,s_{\rm min}}
				\quad
					\forall k\in\N.
				\]
			\end{lemma}
			\begin{proof}
				Set \(s=\min\{\gamma_0,s_{\rm min}\}\).
				Only the case \(\alpha>1\) is nontrivial, since at \(\alpha=1\) the first term of the recurrence is at least \(\gamk\) and the claim follows by induction.
				The claims for \(k=0,1\) follow from \(\rho_0\geq1\).
				To arrive at a contradiction, let \(k\geq1\) be the first index with \(\gamk*<s/\sqrt\alpha\).
				Since \(\gamk\geq s/\sqrt\alpha\) and \(s_k\geq s\), the recurrence gives
				\[
					\gamk*\geq\gamk\sqrt{1/\alpha+\rhok},
					\qquad \rhok<1-1/\alpha\leq1/2.
				\]
				If \(\gamk\geq s\), the first inequality is a contradiction.
				Otherwise \(\gamk<s\), while \(\rhok<1/2\) already gives \(\gamk<\gamk'\).
				Applying the recurrence at \(k-1\) then rules out its second and third terms, so that \(\rhok\geq\sqrt{1/\alpha+\rhok'}\geq1/\sqrt\alpha>1/2\), again a contradiction.
			\end{proof}

			\begin{lemma}\label{thm:gamavg}%
				In the setting of \cref{thm:gammin}, suppose that \(\seq{\gamk}\) satisfies the stronger inequality
				\[
					\gamk*
				\geq
					\min\set*{\gamk\sqrt{\nicefrac{1}{\alpha}+\rhok},\, \gamk\sqrt{\alpha},\, s_k}.
				\]
				Then,
				\[
				\textstyle
					\tfrac{1}{k+1}\sum_{j=0}^k\gamma_j\geq\min\set{\gamma_0,s_{\rm min}}
				\qquad
					\forall k\in\N.
				\]
			\end{lemma}
			\begin{proof}
				Replacing \(s_{\rm min}\) with \(\min\set{\gamma_0,s_{\rm min}}\) if necessary, we may assume that \(\gamma_0\geq s_{\rm min}\) without loss of generality.
				Thus, we shall proceed by induction on \(k\) to show that
				\(
					\Sigma_k
				\coloneqq
					\sum_{j=0}^k\gamma_j\geq (k+1)s_{\rm min}
				\).

				The cases \(k=0\) and \(k=1\) hold by initialization.
				Suppose that the claim holds up to iteration \(k\geq1\), and let \(i=i_{k+1}\geq 0\) denote the maximum number of consecutive indices for which
				\begin{equation}\label{eq:ik+1:def}
					\gamma_{k+1-j}
				<
					s_{\rm min}
				\quad
					\text{holds for all \(j=0,\dots,i-1\).}
				\end{equation}
				Since \(\rho_0,\alpha\geq1\), one has \(\gamma_1\geq\min\set{\gamma_0\sqrt{\nicefrac{1}{\alpha}+\rho_0},\gamma_0\sqrt{\alpha},s_0}\geq s_{\rm min}\), so that \(\gamma_0,\gamma_1\geq s_{\rm min}\) and thus \(i\leq k\).
				In particular,
				\begin{equation}\label{eq:ik+1:rhogeq}
					\rho_{k+1-j}
				\geq
					\min\set*{
						\sqrt{\nicefrac{1}{\alpha}+\rho_{k-j}}
					,\,
						\sqrt{\alpha}
					}
				\quad
					\text{holds for all \(j=0,\dots,i-1\).}
				\end{equation}
				Clearly, if \(i=0\) then \(\gamk*\geq s_{\rm min}\) and the induction step is straightforward.
				Otherwise, considering the last \(i\geq 1\) iterates, if there exists \(0\leq \hat\jmath\leq i-1\) such that \(\rho_{k+1-\hat\jmath}\geq\sqrt{\alpha}\), then  \(\gamma_{k+1-\hat\jmath}\geq\sqrt{\alpha}\gamma_{k-\hat\jmath}\geq s_{\rm min}\), where the last inequality follows from the global lower bound on the stepsizes given in \cref{thm:gammin}.
				This contradicts \eqref{eq:ik+1:def} with \(j=\hat\jmath\).

				Hence \(\rho_{k+1-j}<\sqrt{\alpha}\) holds for all \(j=0,\dots,i-1\), whence
				\begin{equation}\label{eq:ik+1:rhogeq'}
					\rho_{k+1-j}
				\geq
					\sqrt{\nicefrac{1}{\alpha}+\rho_{k-j}}
				\quad
					\text{holds for all \(j=0,\dots,i-1\).}
				\end{equation}
				We consider two subcases:
				\begin{itemize}
				\item
					If \(i\leq 3\), then we shall show that the sum of the last \(i+2\) terms is larger than \((i+2)s_{\rm min}\);
					the induction step will then immediately follow.
					To this end, observe that, by maximality of \(i\) in \eqref{eq:ik+1:def}, \(\gamma_{k-i+1}\geq s_{\rm min}\), and thus
					\begin{align*}
						\sum_{j=k-i}^{k+1}
						\gamma_j
					={} &
						\gamma_{k-i+1}
						\Bigl(
							\tfrac{1}{\rho_{k-i+1}}
							+
							1
							+
							\sum_{j=2}^{i+1}
							\prod_{\ell=2}^j
							\rho_{k-i+\ell}
						\Bigr)
					\geq
						s_{\rm min}
						\Bigl(
							\tfrac{1}{\rho_{k-i+1}}
							+
							1
							+
							\sum_{j=2}^{i+1}
							\prod_{\ell=2}^j
							\rho_{k-i+\ell}
						\Bigr).
					\end{align*}
					Let \(\phi_\alpha(t)\coloneqq\sqrt{\nicefrac{1}{\alpha}+t}\), so that \eqref{eq:ik+1:rhogeq'} can be equivalently cast as \(\rho_j\geq\phi_\alpha(\rho_{j-1})\) for \(j=k-i+2,\dots,k+1\).%

					\vspace*{-.5\baselineskip}%
					\noindent
					\begin{minipage}[b]{0.63\linewidth}%
						Notice further that \(\phi_\alpha^2(t)\geq\phi_\alpha^2(0)=\sqrt{\nicefrac{1}{\alpha}+\sqrt{\nicefrac{1}{\alpha}}}\geq1\) and \(t\leq\phi_\alpha(t)\leq\phi_\alpha^2(t)\leq\dots\leq\phi_\alpha^\ell(t)\) for any \(t\leq\rhomax\), while \(\phi_\alpha^\ell(t)\geq\rhomax\geq1\) for any \(t\geq\rhomax\), where \(\phi_\alpha^\ell\) denotes \(\phi_\alpha\) composed with itself \(\ell\) times and \(\rhomax=\frac{1}{2}\bigl(1+\sqrt{1+\frac{4}{\alpha}}\bigr)\) is the positive fixed point of \(\phi_\alpha\).
						Hence, \(\phi_\alpha^\ell(t)\geq 1\) for \(\ell\geq2\) and \(t>0\), and the above inequality can be worked into
					\end{minipage}
					\hfill
					\begin{minipage}[b]{0.32\linewidth}%
						\centering
						\fbox{%
							\includetikz[width=.93\linewidth]{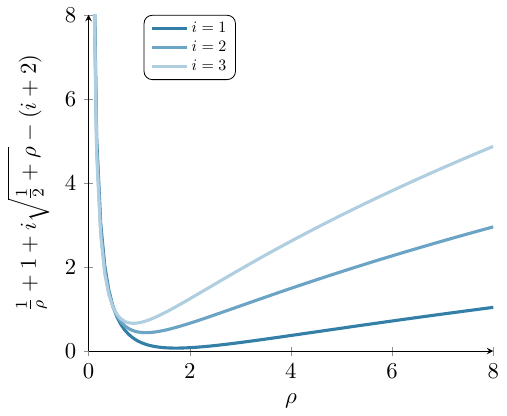}%
						}%
					\end{minipage}%

					\noindent
					\begin{align*}
						\sum_{j=k-i}^{k+1}
						\gamma_j
					\geq{} &
						s_{\rm min}
						\Bigl(
							\tfrac{1}{\rho_{k-i+1}}
							+
							1
							+
							\sum_{j=2}^{i+1}
							\prod_{\ell=2}^j
							\phi_\alpha^{\ell-1}(\rho_{k-i+1})
						\Bigr)
					\\
					\geq{} &
						s_{\rm min}
						\Bigl(
							\tfrac{1}{\rho_{k-i+1}}
							+
							1
							+
							\phi_\alpha(\rho_{k-i+1})
							+
							\sum_{j=3}^{i+1}
							\phi_\alpha(\rho_{k-i+1})
							\prod_{\ell=3}^j
							\overbracket*[0.5pt]{
								\phi_\alpha^{\ell-1}(\rho_{k-i+1})
							}^{\geq1}
						\Bigr)
					\\
					\geq{} &
						s_{\rm min}
						\Bigl(
							\tfrac{1}{\rho_{k-i+1}}
							+
							1
							+
							i
							\phi_\alpha(\rho_{k-i+1})
						\Bigr)
					\geq
						s_{\rm min}
						\inf_{\rho>0}
						\Bigl(
							\tfrac{1}{\rho}
							+
							1
							+
							i
							\phi_2(\rho)
						\Bigr)
					\geq
						(i+2)s_{\rm min},
					\end{align*}
					where the second-last inequality uses the fact that \(\phi_\alpha\geq\phi_2\) for any \(\alpha\in[1,2]\), and the last follows for \(1\leq i\leq3\) since \(\rho^{-1}\geq2\) and \(i(1-1/\sqrt2)<1\) when \(\rho\leq\frac12\), while \(\rho^{-1},\phi_2(\rho)\geq1\) when \(\frac12\leq\rho\leq1\).
					For \(\rho\geq1\), writing \(u=\rho-1\), the identity
					\[
						\rho^2\bigl(\rho+\tfrac12-(2-\rho^{-1})^2\bigr)
						=u^3-\tfrac12u^2+\tfrac12>0
					\]
					gives \(\phi_2(\rho)\geq2-\rho^{-1}\) and hence the same bound.

				\item
					If \(i\geq 4\), then
					\(
						\prod_{j=k-i+2}^{k+1}\rho_j
					\geq
						\prod_{\ell=1}^i
						\phi_\alpha^\ell(\rho_{k-i+1})
					\geq
						\prod_{\ell=1}^4
						\phi_2^\ell(0)
					=
						1.3047\ldots
					\geq
						1
					\),
					where the factors with \(\ell>4\) are at least \(1\) as shown above.
					Combined with \(\gamma_{k-i+1}\geq s_{\rm min}\), which holds by maximality of \(i\), we obtain that
					\(
						\gamk*
					=
						\gamma_{k-i+1}
						\prod_{j=k-i+2}^{k+1}\rho_j
					\geq
						s_{\rm min}
					\),
					which contradicts \eqref{eq:ik+1:def} with \(j=0\). Hence this case cannot occur.
				\qedhere
				\end{itemize}
			\end{proof}

		\subsection{Proof of \texorpdfstring{\cref{thm:main:innprods}}{Lemma \ref{thm:main:innprods}}}\label{proof:thm:main:innprods}%
			Consider the iterates \eqref{eq:xk*} with \(\uk*=F(\xk)+\thetk*\DFk\) for an arbitrary reflection coefficient \(\thetk*\geq0\), so that \eqref{eq:alg-family} is recovered with \(\thetk*=\alpha\rhok*\).
			We take \(k\geq1\), so that \(\xk\) is a proximal iterate and, since \(x^0\in\dom g\), also \(\xk'\in\dom g\).
			Fix \(x\in\dom g\) and write \(\Vk\) in place of \(\Vk(x)\).
			For \(y\in\dom g\) we denote
			\(
				\Dg(y,\xk)
			\coloneqq
				g(y)-g(\xk)-\innprod{\nabla*g(\xk)}{y-\xk}
			\geq
				0
			\),
			where \(\nabla*g(\xk)=-\frac{1}{\gamk}\Dxk-\uk\) is as in \eqref{eq:tildeg}.
			Let \(\wk*\) be as in \cref{thm:FNE}.
			Its definition and \eqref{eq:tildeg} give \(-\nabla*g(\xk)=-\frac{1}{\gamk}\wk*+\uk*\).
			Then
			\begin{align*}
				0
			\leq{} &
				\Dg(x,\xk*)+\Dg(\xk*,\xk)
			\\
			={} &
				g(x)-g(\xk)
				+
				\innprod{\tfrac{1}{\gamk*}\Dxk*+\uk*}{x-\xk*}
				-
				\innprod{\tfrac{1}{\gamk}\wk*-\uk*}{\Dxk*}
			\\
			={} &
				\innprod{\uk*-F(\xk)}{x-\xk}
				-
				\Vk
				+
				\tfrac{1}{\gamk*}
				\innprod{\Dxk*}{x-\xk*}
				-
				\tfrac{1}{\gamk}
				\innprod{\wk*}{\Dxk*}
			\\
			={} &
				\innprod{\uk*-F(\xk)}{x-\xk}
				-
				\Vk
				-
				\tfrac{1}{\gamk}
				\innprod{\wk*}{\Dxk*}
			\\
			&
				+
				\tfrac{1}{2\gamk*}
				\bigl(
					\norm{\xk-x}^2
					-
					\norm{\xk*-x}^2
					-
					\norm{\Dxk*}^2
				\bigr),
			\numberthis\label{eq:Dg-first}
			\end{align*}
			where the second equality uses \eqref{eq:Vk} and the third uses the identity \(2\innprod{a}{b}=\norm{a+b}^2-\norm{a}^2-\norm{b}^2\).

			Similarly, \(0\leq\Dg(\xk',\xk)\) gives
			\begin{align*}
				0
			\leq{} &
				g(\xk')
				-
				g(\xk)
				-
				\innprod{\tfrac{1}{\gamk}\Dxk+\uk}{\Dxk}
			\\
			={} &
				\Vk'-\Vk
				+
				\innprod{F(\xk')-\uk}{\Dxk}
				-
				\innprod{\DFk}{x-\xk}
				-
				\tfrac{1}{\gamk}
				\norm{\Dxk}^2
			\\
			={} &
				\Vk'-\Vk
				-
				\thetk
				\innprod{\Dxk}{\DFk'}
				-
				\innprod{\DFk}{x-\xk}
				-
				\tfrac{1}{\gamk}
				\norm{\Dxk}^2,
			\numberthis\label{eq:Dg-second}
			\end{align*}
			where the first equality uses \eqref{eq:Vk} for \(\Vk'\) and \(\Vk\), and the second uses \(F(\xk')-\uk=-\thetk\DFk'\).
			We multiply \eqref{eq:Dg-first} by \(\gamk*\) and add \eqref{eq:Dg-second} multiplied by \(\thetk*\gamk*\).
			The inner products against \(x-\xk\) combine into \(\gamk*\innprod{\uk*-\bigl[F(\xk)+\thetk*\DFk\bigr]}{x-\xk}\), which is null by the choice of \(\uk*\).
			Bounding \(-\rhok*\innprod{\wk*}{\Dxk*}\leq\rhok*^2\norm{\wk*}^2\) by \cref{thm:FNE} and rearranging gives
			\begin{align*}
			&
				\tfrac{1}{2}
				\norm{\xk*-x}^2
				-
				\tfrac{1}{2}
				\norm{\xk-x}^2
				+
				\tfrac{1}{2}
				\norm{\Dxk*}^2
				+
				\gamk*(1+\thetk*)\Vk
			\\
			\leq{} &
				\thetk*\gamk*\Vk'
				-
				\thetk*\rhok*
				\norm{\Dxk}^2
				-
				\thetk\thetk*\gamk*
				\innprod{\Dxk}{\DFk'}
				+
				\rhok*^2
				\norm{\wk*}^2.
			\numberthis\label{eq:main_ineq}
			\end{align*}
			It remains to expand \(\norm{\wk*}^2\).
			Since \(\uk-\uk*=\thetk\DFk'-(1+\thetk*)\DFk\),
			\[
				\wk*
			=
				-\bigl(\Dxk-(1+\thetk*)\gamk\DFk\bigr)
				-
				\thetk\gamk\DFk'.
			\]
			We expand the square, use \(\norm{\DFk}=\Lk\norm{\Dxk}\) as in \eqref{eq:Lk}, and discard the term \(-2(1+\thetk*)\gamk\innprod{\DFk}{\Dxk}\), which is nonpositive by monotonicity of \(F\). This gives
			\begin{align*}
				\norm{\wk*}^2
			\leq{} &
				\bigl[(1+\thetk*)^2\gamk^2\Lk^2+1\bigr]\norm{\Dxk}^2
				+
				\thetk^2\gamk^2\norm{\DFk'}^2
			\\
			&
				+
				2\thetk\gamk\innprod{\Dxk}{\DFk'}
				-
				2\thetk(1+\thetk*)\gamk^2\innprod{\DFk}{\DFk'}.
			\end{align*}
			Plugging this into \eqref{eq:main_ineq} and using \(\gamk*=\rhok*\gamk\) to collect the terms in \(\norm{\Dxk}^2\) and \(\innprod{\Dxk}{\DFk'}\) yields
			\begin{align*}
			&
				\tfrac{1}{2}
				\norm{\xk*-x}^2
				+
				\tfrac{1}{2}
				\norm{\Dxk*}^2
				+
				\gamk*(1+\thetk*)\Vk
			\\
			\leq{} &
				\tfrac{1}{2}
				\norm{\xk-x}^2
				+
				\thetk*\gamk*\Vk'
				+
				\thetk^2\gamk*^2
				\norm{\DFk'}^2
			\\
			&
				+
				\rhok*^2
				\biggl[
					(1+\thetk*)^2\gamk^2\Lk^2+1
					-
					\tfrac{\thetk*}{\rhok*}
				\biggr]
				\norm{\Dxk}^2
			\\
			&
				+
				\thetk\gamk\rhok*(2\rhok*-\thetk*)
				\innprod{\Dxk}{\DFk'}
				-
				2\thetk(1+\thetk*)\gamk^2
				\rhok*^2
				\innprod{\DFk}{\DFk'}.
			\numberthis\label{eq:main:general}
			\end{align*}
			Setting \(\thetk=\alpha\rhok\) and \(\thetk*=\alpha\rhok*\) yields the claim.

	\appendixnotitle \label{sec:extended}%
		\subsection{Proof of \texorpdfstring{\cref{ex:sharpness}}{Example \ref{ex:sharpness}}}\label{proof:ex:sharpness}%
			With \(w^k\coloneqq(\xk,\xk')\in\R^4\), the iteration \(\xk*=\xk-\gamma\bigl((1+\alpha)S\xk-\alpha S\xk'\bigr)\) reads \(w^{k+1}=Mw^k\), where
			\[
				M
			\coloneqq
				\begin{pmatrix}
					I-\gamma(1+\alpha)S & \gamma\alpha S
				\\
					I & 0
				\end{pmatrix}.
			\]
			Since the blocks of \(M\) commute and \(\det(aI+bS)=a^2+b^2\) for \(a,b\in\R\),
			\[
				\det(\lambda I-M)
			=
				\det\Bigl((\lambda^2-\lambda)I+\gamma\bigl((1+\alpha)\lambda-\alpha\bigr)S\Bigr)
			=
				(\lambda^2-\lambda)^2+\gamma^2\bigl((1+\alpha)\lambda-\alpha\bigr)^2.
			\]
			This quartic is a sum of two squares \(A(\lambda)^2+B(\lambda)^2\) of the real polynomials \(A(\lambda)\coloneqq\lambda^2-\lambda\) and \(B(\lambda)\coloneqq\gamma\bigl((1+\alpha)\lambda-\alpha\bigr)\).
			It has no convenient real factorization, but over \(\C\) it factors as \((A+iB)(A-iB)\).
			We therefore study the complex quadratic
			\[
				p(\lambda)
			\coloneqq
				A(\lambda)+iB(\lambda)
			=
				\lambda^2+a_1\lambda+a_0,
			\]
			where \(a_1=(1+\alpha)i\gamma-1\) and \(a_0= -\alpha i\gamma\).
			The other factor is \(A-iB=\lambda^2+\bar a_1\lambda+\bar a_0\).
			Since \(A\) and \(B\) have real coefficients, \(\lambda\) is a root of \(A+iB\) if and only if \(\bar\lambda\) is a root of \(A-iB\).
			The eigenvalues of \(M\) are thus the roots of \(p\) and their complex conjugates, which have the same moduli.
			It therefore suffices to locate the roots of \(p\).

			If the roots of \(p\) lie in \(\set{\abs{\lambda}<1}\), then \(M\) has spectral radius less than \(1\) and \(w^k\to0\) for every initialization.
			By the Schur--Cohn test, this holds if and only if \(\abs{a_0}<1\) and \(\abs{a_1-\bar a_1a_0}<1-\abs{a_0}^2\).
			We have \(\abs{a_0}^2=\alpha^2\gamma^2\) and \(a_1-\bar a_1a_0=\bigl(\alpha(1+\alpha)\gamma^2-1\bigr)+i\gamma\), so that
			\[
				\bigl(1-\abs{a_0}^2\bigr)^2
				-
				\bigl|a_1-\bar a_1a_0\bigr|^2
			=
				\gamma^2
				\Bigl[
					(2\alpha-1)
					-
					\gamma^2\alpha^2(1+2\alpha)
				\Bigr].
			\]
			For \(\gamma>0\) this is positive exactly when \(\alpha>\tfrac12\) and \(\gamma^2<\frac{2\alpha-1}{\alpha^2(1+2\alpha)}\).
			This bound implies \(\alpha^2\gamma^2<\frac{2\alpha-1}{2\alpha+1}<1\), so it also implies \(\abs{a_0}<1\).

			Suppose now that the bound fails, and let \(z=u+iv\) be a root of \(p\) with \(\abs{z}\geq1\).
			With \(Z\coloneqq uI+vS\), consider the sequence
			\[
				s^k
			\coloneqq
				Z\xk+\alpha\gamma S\xk'
			\]
			which is chosen so that \(s^{k+1}=Zs^k\).
			Indeed, expanding with \(S^2=-I\) and comparing the coefficients of \(I\) and \(S\) with the real and imaginary parts of \(p(z)=0\) gives \(Z^2=Z-\gamma\bigl((1+\alpha)Z-\alpha I\bigr)S\).
			Substituting the update for \(x^{k+1}\) into \(s^{k+1}=Z\xk*+\alpha\gamma S\xk\) and using this identity together with \(ZS=SZ\) yields \(s^{k+1}=Zs^k\).
			Since \(Z^\top Z=\abs{z}^2I\), this gives \(\norm{s^k}=\abs{z}^k\norm{s^0}\geq\norm{s^0}\).
			As \(\norm{s^k}\leq\abs{z}\norm{\xk}+\alpha\gamma\norm{\xk'}\), it follows that \(\xk\not\to0\) whenever \(s^0\neq0\).
			Finally, \(x^{-1}=x^0\neq0\) gives \(s^0=\bigl(uI+(v+\alpha\gamma)S\bigr)x^0 \neq 0\), where the matrix has determinant \(\abs{z+i\alpha\gamma}^2\).
			The determinant is nonzero because \(z\neq-i\alpha\gamma\).
			For \(\alpha>0\) this holds since \(p(-i\alpha\gamma)=\alpha\gamma^2\neq0\), and for \(\alpha=0\) since \(-i\alpha\gamma=0\) while \(\abs{z}\geq1\).

		\subsection{Proof of \texorpdfstring{\cref{thm:Lyapunov}}{the Lyapunov-type descent lemma}}\label{proof:thm:Lyapunov}%
			Fix \(x\in\dom g\). We start by writing the inequality in \cref{thm:main:innprods}, namely
			\begin{align*}
			&
				\tfrac{1}{2}
				\norm{\xk*-x}^2
				+
				\tfrac{1}{2}
				\norm{\Dxk*}^2
				+
				\gamk*(1+\alpha\rhok*)\Vk(x)
			\\
			\leq{} &
				\tfrac{1}{2}
				\norm{\xk-x}^2
				+
				\alpha\rhok*\gamk*\Vk'(x)
				+
				(\alpha\rhok\gamk*)^2
				\norm{\DFk'}^2
			\\
			&
				+
				\rhok*^2
				\Bigl[
					(1+\alpha\rhok*)^2\gamk^2\Lk^2
					+
					1
					-
					\alpha
				\Bigr]
				\norm{\Dxk}^2
			\\
			&
				+
				\alpha\rhok\gamk\rhok*^2(2-\alpha)
				\innprod{\Dxk}{\DFk'}
			\\
			&
				-
				2\alpha\rhok(1+\alpha\rhok*)\gamk*^2
				\innprod{\DFk}{\DFk'}.
			\end{align*}
			We then substitute the identity \eqref{eq:tau-decomposition} for the first inner product and the bound \eqref{eq:cosphi-bound-analysis} for the second one, resulting in
			\begin{align*}
			&
				\tfrac{1}{2}
				\norm{\xk*-x}^2
				+
				\tfrac{1}{2}
				\norm{\Dxk*}^2
				+
				\gamk*(1+\alpha\rhok*)\Vk(x)
			\\
			\leq{} &
				\tfrac{1}{2}
				\norm{\xk-x}^2
				+
				\alpha\rhok*\gamk*\Vk'(x)
				+
				(\alpha\rhok\gamk*)^2
				\biggl\{
					1
					+
					\tauk
					\tfrac{2-\alpha}{2\alpha\varepsilon}
					+
					\tfrac{[\cos\phi_k]_-}{\mu}
				\biggr\}
				\norm{\DFk'}^2
			\\
			&
				+
				\rhok*^2
					\begin{array}[t]{@{} c @{} r @{} c @{}}
						\biggl\{{}
					&
						(\gamk\Lk)^2
						\bigl(
							1
							+
							\mu
							[\cos\phi_k]_-
						\bigr)
						(1+\alpha\rhok*)^2
						+
						1
						-
						\alpha
						+
						\tauk\alpha\varepsilon
						\tfrac{2-\alpha}{2}
					&
						{}\biggr\}
						\norm{\Dxk}^2.
					\end{array}
			\end{align*}
			In terms of \(\Uk(x)\) as in the statement it reads
			\begin{align*}
			&
				\Uk*(x)-\Uk(x)
				+
				\alpha\bigl(\tfrac{1}{\alpha}+\rhok-\rhok*^2\bigr)
				\gamk\Vk'(x)
			\\
			\leq{} &
				-
				(\alpha\rhok\gamk)^2
				\Bigl\{
					\tfrac{\lambda}{2\alpha^2}
					-
					\rhok*^2
					\Bigl[
						1
						+
						\tfrac{(2-\alpha)\tauk}{2\alpha\varepsilon}
						+
						\tfrac{[\cos\phi_k]_-}{\mu}
					\Bigr]
				\Bigr\}
				\norm{\DFk'}^2
			\\
			&
				-
				\biggl\{
					\tfrac{1}{2}
					-
					\rhok*^2
					\biggl[
						\tfrac{\lambda(\gamk*\Lk)^2}{2}
						+
						(1+\alpha\rhok*)^2
						(\gamk\Lk)^2
						\bigl(
							1
							+
							\mu
							[\cos\phi_k]_-
						\bigr)
					\\
					&
					\qquad
					\qquad
						+
						1
						-
						\alpha
						+
						(2-\alpha)
						\tfrac{\tauk\alpha\varepsilon}{2}
					\biggr]
				\biggr\}
				\norm{\Dxk}^2.
			\end{align*}
			This is the claimed inequality.

		\subsection{\texorpdfstring{\refAdaFRB+: the general reflection parameter}{adaFRB+a: the general reflection parameter}}\label{sec:adaFRB+}%
			For simplicity, \refadaFRBII{} in \cref{alg:adaFRB+2} is restricted to \(\alpha=2\).
			This section extends it to every \(\alpha\in[1,2]\).
			For \(\alpha<2\) the factor \(\tauk\) does not drop out of \eqref{eq:ABC}.
			Retaining \(\tauk\) and \([\cos\phi_k]_-\) in \eqref{eq:ABC}, instead of bounding them by \(1\) as done in \eqref{eq:ABCwc}, leads to \refadaFRB+ in \cref{alg:adaFRB+}.
			The values of \(\varepsilon\), \(\mu\) and \(\lambda\) set at initialization, namely
			\begin{equation}\label{eq:adaFRB+:coefficients}
				\bar c
			=
				\varepsilon
			=
				\tfrac{
					2\alpha-1
				}{
					\alpha(2-\alpha)
					+
					\sqrt{\alpha^2(2-\alpha)^2+2(2\alpha-1)(2\alpha+1)^2}
				}
			,\quad
				\mu=\tfrac{\alpha}{1+\alpha}
			,\quad\text{and}\quad
				\lambda
			=
				2\alpha^2
				\bigl(
					1
					+
					\tfrac{1}{\mu}
					+
					\tfrac{2-\alpha}{2\alpha\varepsilon}
				\bigr),
			\end{equation}
			coincide with those derived in the convergence analysis of the nonadaptive \refadaFRB{} as in \cref{proof:thm:FRBa}.
			The bound \(\bar c\) equals \(c(\alpha)\) in \eqref{eq:FRBa_gamma}, so that \refadaFRB+ attains the constant-stepsize threshold of \cref{thm:FRBa} for every \(\alpha\in[1,2]\).
			As \(\alpha\) ranges in \([1,2]\), \(\bar c\) increases from \(\frac{\sqrt{19}-1}{18}\approx0.187\) to \(\frac{\sqrt{6}}{10}\approx0.245\), whereas the cap \(c=\frac{1}{7-\alpha}\) of \refadaFRB{} increases from \(\frac{1}{6}\approx0.167\) to \(\frac{1}{5}=0.2\).
			The price is a ratio cap that may equal \(1\) in a worst-case scenario, namely
			\[
				\beta_k
			=
				\min\set*{
					\sqrt{
						\tfrac{
							1
							+
							\frac{1}{\mu}
							+
							\frac{2-\alpha}{2\alpha\varepsilon}
						}{
							\bigl[
								1
								+
								\frac{1}{\mu}
								[\cos\phi_k]_-
								+
								\tauk
								\frac{2-\alpha}{2\alpha\varepsilon}
							\bigr]_+
						}
					}
					,\,
					\tilde\rho_k
				}
			\]
			where \(\tilde\rho_k\) is the smallest positive root of
			\begin{equation}\label{eq:C}
			\begin{split}
				q_k(\rho)
			\coloneqq{} &
				\Bigl[
					2(\alpha-1)
					-
					\tauk\alpha\varepsilon(2-\alpha)
					-
					2(1+\mu[\cos\phi_k]_-)\alpha^2\bar c^2
					-
					\lambda\bar c^2
				\Bigr]
				\rho^2
			\\
			&
				-
				\Bigl[
					4(1+\mu[\cos\phi_k]_-)\alpha\bar c^2
				\Bigr]
				\rho
				+
				\Bigl[
					1
					-
					2(1+\mu[\cos\phi_k]_-)\bar c^2
				\Bigr],
			\end{split}
			\end{equation}
			with \(\tilde\rho_k=\infty\) if there is none (in which case the second-order polynomial is always positive).

			When \(\alpha=2\), all the terms carrying the factor \(2-\alpha\) vanish, and the first term of \(\betk\) becomes \(\sqrt{\nicefrac{5}{(2+3[\cos\phi_k]_-)}}\).
			Then \(q_k=\frac{2}{25}q_t\) with \(t=[\cos\phi_k]_-\) and \(q_t\) as in \eqref{eq:adaFRB+2:qt}, and the proof of \cref{thm:adaFRB} shows that \(q_t\) is positive on \([0,\betk)\).
			Hence this term never exceeds \(\tilde\rho_k\), so \(\betk\) simplifies to \(\sqrt{\nicefrac{5}{(2+3[\cos\phi_k]_-)}}\) and \refadaFRB+ reduces to \refadaFRBII.

			\begin{algorithm}[tb]%
				\caption{\AdaFRB+}%
				\label{alg:adaFRB+}%
					\begin{algorithmic}[1]
					\item[{%
						Choose \({\alpha\in[1,2]}\), \(x^0\in\dom g\), \(\gamma_0>0\), \(L_0>0\), and \(c\in(0,\bar c)\) with \(\bar c\) as below%
					}]
					\item[{%
						Set
						\(x^{-1}=x^{-2}=x^0\),~
						\(\rho_0=1\),
					}]
					\item[{%
						\hphantom{Set }%
						\(
							\bar c
						=
							\varepsilon
						=
							\tfrac{
								2\alpha-1
							}{
								\alpha(2-\alpha)
								+
								\sqrt{\alpha^2(2-\alpha)^2+2(2\alpha-1)(2\alpha+1)^2}
							}
						\),~
						\(
							\mu=\tfrac{\alpha}{1+\alpha}
						\),~
						and~
						\(
							\lambda
						=
							2\alpha^2
							\bigl(
								1
								+
								\tfrac{1}{\mu}
								+
								\tfrac{2-\alpha}{2\alpha\varepsilon}
							\bigr)
						\)
					}]
					\item[{%
						Repeat for \(k=0,1,\dots\)%
					}]
					\itemsep=5pt
					\State \label{state:adaFRB+:betak}%
						\(
							[\cos\phi_k]_-
						=
							\frac{
								\left[\innprod{\DFk}{\DFk'}\right]_-
							}{
								\norm{\DFk}
								\norm{\DFk'}
							}
						\),\quad
						\(
							\tauk
						=
							\frac{
								\innprod{\Dxk}{\DFk'}
							}{
								\frac{\varepsilon}{2\gamk\rhok}
								\norm{\Dxk}^2
								+
								\frac{\rhok\gamk}{2\varepsilon}
								\norm{\DFk'}^2
							}
						\)%
					{\Comment%
						\(\Dxk\) and \(\DFk\) as in \eqref{eq:Dk}, \(\tfrac{0}{0}=0\)
					}%

					\Statex
						\(
							\beta_k
						=
							\min\set*{
								\sqrt{
									\tfrac{
										1
										+
										\frac{1}{\mu}
										+
										\frac{2-\alpha}{2\alpha\varepsilon}
									}{
										\bigl[
											1
											+
											\frac{1}{\mu}
											[\cos\phi_k]_-
											+
											\tauk
											\frac{2-\alpha}{2\alpha\varepsilon}
										\bigr]_+
									}
								}
								,\,
								\tilde\rho_k
							}
						\)
						\smash{%
							\begin{tabular}[t]{l@{}}
								where \(\tilde\rho_k\in[1,\infty]\) is the smallest positive root
							\\
								of the quadratic \eqref{eq:C} in \(\rho\) (\(\infty\) if none exists)
							\end{tabular}%
						}%

					\State \label{state:adaFRB+:rhok*}%
						\(
							\gamk*
						=
							\min\set*{
								\gamk
								\sqrt{
									\tfrac{1}{\alpha}+\rhok
								}
								,\,
								\beta_k
								\gamk
								,\,
								\tfrac{c}{\Lk}
							}
						\)
						~and~
						\(
							\rhok*
						=
							\frac{\gamk*}{\gamk}
						\)
					{\Comment%
						\(\Lk\) as in \eqref{eq:Lk}%
					}

					\State
						\(
							\uk*
						=
							F(\xk)
							+
							\alpha\rhok*
							\bigl(F(\xk)-F(\xk')\bigr)
						\)

					\State
						\(
							\xk*
						=
							\prox_{\gamk* g}\bigl(
								\xk
								-
								\gamk*\uk*
							\bigr)
						\)
					\end{algorithmic}
				\end{algorithm}

				\begin{theorem}[convergence of {\protect\refadaFRB+}]\label{thm:adaFRB+}%
					Let \cref{ass:f,ass:g,ass:sol} hold, and consider the iterates generated by \refadaFRB+{} (\cref{alg:adaFRB+}), with \(\alpha\) and \(c\) as therein.
					Then, the sequence \(\seq{\xk}\) converges to a solution, and all assertions of \cref{thm:adaFRB} hold with \(\lambda\) in \eqref{eq:Uk} as set in \cref{alg:adaFRB+}.
				\end{theorem}
				\begin{proof}
					Following our template, we shall show that \(A_{k+1},B_{k+1}\geq0\) and \(C_{k+1}\geq C_{\rm min}>0\) for the coefficients defined in \eqref{eq:ABC} with the values of \(\varepsilon,\lambda,\mu\) given at algorithmic initialization.
					The first term in the minimum at \cref{state:adaFRB+:rhok*} enforces \(\rhok*\leq\sqrt{\frac{1}{\alpha}+\rhok}\), which is precisely \(A_{k+1}\geq0\).
					With the values of \(\varepsilon\), \(\lambda\) and \(\mu\) in \cref{alg:adaFRB+}, the condition \(B_{k+1}\geq0\) reads
					\begin{equation}\label{eq:AB}
						\rhok*^2
					\leq
						\tfrac{
							1+\frac{1}{\mu}+\frac{2-\alpha}{2\alpha\varepsilon}
						}{
							1+\frac{1}{\mu}[\cos\phi_k]_-+\tauk\frac{2-\alpha}{2\alpha\varepsilon}
						},
					\end{equation}
					and the first term of \(\betk\) at \cref{state:adaFRB+:betak} enforces exactly this bound, with the convention \(\frac{t}{0}=\infty\) covering the case of a nonpositive denominator, in which \(B_{k+1}\geq0\) holds for every \(\rhok*\).
					Otherwise, the right-hand side is at least \(1\), since \(\tauk\leq1\) and \([\cos\phi_k]_-\leq1\).

					As for \(C_{k+1}\), recall from \eqref{eq:ABC} that
					\begin{align*}
						C_{k+1}
					={} &
						\tfrac{1}{2}
						-
						(\gamk*\Lk)^2
						\Bigl(
							\tfrac{\lambda}{2}\rhok*^2
							+
							(1+\mu[\cos\phi_k]_-)(1+\alpha\rhok*)^2
						\Bigr)
						+
						\Bigl(
							\alpha-1-\tauk\alpha\varepsilon\tfrac{2-\alpha}{2}
						\Bigr)
						\rhok*^2.
					\end{align*}
					The third term in the minimum at \cref{state:adaFRB+:rhok*} enforces \(\gamk*\Lk\leq c\).
					Using this bound and writing \(c^2=\bar c^2-(\bar c^2-c^2)\), we obtain
					\begin{align}
					\nonumber
						C_{k+1}
					\geq{} &
						(\bar c^2-c^2)
						\Bigl(
							\tfrac{\lambda}{2}\rhok*^2
							+
							(1+\mu[\cos\phi_k]_-)(1+\alpha\rhok*)^2
						\Bigr)
					\\
					\label{eq:Clb}
					&
						+
						\tfrac{1}{2}
						q_k(\rhok*),
					\end{align}
					as one verifies by expanding \((1+\alpha\rhok*)^2\).
					The expression in parentheses in the first line is at least \((1+\alpha\rhok*)^2\geq1\), so the first line is at least \(\bar c^2-c^2>0\).
					The constant term of \(q_k\) is positive, because \(1+\mu[\cos\phi_k]_-\leq2\) and \(\bar c^2\leq\frac{6}{100}\).
					Hence \(q_k\) is nonnegative on \([0,\tilde\rho_k]\), where, as prescribed in \cref{state:adaFRB+:betak}, \(\tilde\rho_k\) is the smallest positive root of the quadratic \eqref{eq:C} in \(\rho\), with \(\tilde\rho_k=\infty\) if there is none (in this case, the second-order polynomial is always positive).
					Since \(\rhok*\leq\betk\leq\tilde\rho_k\), it follows that \(q_k(\rhok*)\geq0\), and therefore \(C_{k+1}\geq\bar c^2-c^2>0\).

					It remains to show that \(\tilde\rho_k\geq1\), which together with the bound on \eqref{eq:AB} gives \(\betk\geq1\).
					In \(q_k\), \(\tauk\) enters only through the term \(-\tauk\alpha\varepsilon(2-\alpha)\rho^2\), and \([\cos\phi_k]_-\) enters only through the terms
					\(
						-2\mu[\cos\phi_k]_-\bar c^2(\alpha^2\rho^2+2\alpha\rho+1)
					=
						-2\mu[\cos\phi_k]_-\bar c^2(1+\alpha\rho)^2
					\).
					For \(\rho\geq0\) both terms are nonincreasing in \(\tauk\) and in \([\cos\phi_k]_-\).
					Since \(\tauk\leq1\) and \([\cos\phi_k]_-\leq1\), \(q_k(\rho)\) is therefore bounded below for every \(\rho\geq0\) by its value in the worst case \(\tauk=[\cos\phi_k]_-=1\), namely
					\begin{equation}\label{eq:Cworst}
						\Bigl(
							2(\alpha-1)
							-
							\alpha\varepsilon(2-\alpha)
							-
							2(1+\mu)\alpha^2\bar c^2
							-
							\lambda\bar c^2
						\Bigr)
						\rho^2
						-
						4(1+\mu)\alpha\bar c^2\rho
						+
						1
						-
						2(1+\mu)\bar c^2.
					\end{equation}
					It thus suffices to show that \eqref{eq:Cworst} is positive for \(\rho\in[0,1)\).
					Let \(d=1-2(1+\mu)\bar c^2\) and \(b=4(1+\mu)\alpha\bar c^2\) denote the constant term and the negated coefficient of \(\rho\) in \eqref{eq:Cworst}.
					Since \(\varepsilon\), \(\mu\) and \(\lambda\) coincide with the values chosen in \cref{proof:thm:FRBa} and \(\bar c=c(\alpha)\), the value of \eqref{eq:Cworst} at \(\rho=1\) is twice the lower bound on \(C_{k+1}\) therein with \(\gamma L_F=\bar c\), which is zero.
					Hence the coefficient of \(\rho^2\) in \eqref{eq:Cworst} equals \(b-d\), and \eqref{eq:Cworst} factors as
					\[
						(b-d)\rho^2-b\rho+d
					=
						(1-\rho)\bigl(d-(b-d)\rho\bigr).
					\]
					For \(\rho\in[0,1)\) the first factor is positive.
					The second factor is linear in \(\rho\), with value \(d\) at \(\rho=0\) and \(2d-b\) at \(\rho=1\).
					Both values are positive, since \(1+\mu\leq2\) gives \(d\geq1-4\bar c^2>0\), and \((1+\mu)(1+\alpha)=1+2\alpha\leq5\) gives \(2d-b=2-4(1+2\alpha)\bar c^2\geq2-20\bar c^2>0\), where we used \(\bar c^2\leq\frac{6}{100}\).
					The second factor is therefore positive on \([0,1]\).
					Hence \eqref{eq:Cworst} is positive on \([0,1)\), and so is \(q_k\), which proves \(\tilde\rho_k\geq1\).
					The cap \(\bar c\) is thus dictated by \(C_{k+1}\geq0\) at \(\rhok*=1\) in the worst case, whereas \refadaFRB{} requires \(C_{k+1}\geq0\) up to the ratio bound \(\beta>1\) in \eqref{eq:adaFRB:beta}.

					In summary, \(A_{k+1},B_{k+1}\geq0\), \(C_{k+1}\geq\bar c^2-c^2>0\) and \(\betk\geq1\) for every \(k\in\N\).
					These are the only properties used in the concluding paragraph of \cref{proof:thm:adaFRB}, which therefore applies verbatim with \cref{state:adaFRB+:rhok*} from \refadaFRB+ in place of \cref{state:adaFRB+2:gamk*} from \refadaFRBII{} and \(C_{\rm min}=\bar c^2-c^2\).%
				\end{proof}

			\subsection{Compared methods}\label{sec:sim_methods}%
				\begin{itemize}[leftmargin=*]
				\item
					\textbf{Forward-reflected-backward (\frb) and the proposed generalization \FRB}

					Proposed by Malitsky and Tam \cite{malitsky2020forward}, differently from the extragradient and forward-backward-forward methods it requires only one evaluation of the operator \(F\) per iteration.
					This method corresponds to \FRB{} with the specific choice \(\alpha=1\).
					For any fixed \(\alpha>\nicefrac{1}{2}\), starting from the current iterate \(\xk\) one iteration of the generalized \FRB{} amounts to
					\[
						\xk*
					=
						\prox_{\gamma g}\Bigl(
							\xk
							-
							\gamma\bigl[
								(1+\alpha)F(\xk)
								-
								\alpha F(\xk')
							\bigr]
						\Bigr).
					\]
					When \(\alpha=1\), the tailored analysis of \cite{malitsky2020forward} permits any \(\gamma<\nicefrac{1}{(2L_F)}\), less restrictive than our bound \eqref{eq:FRBa_gamma} in the constant-stepsize regime.
					We therefore use a stepsize consistent with \cite{malitsky2020forward} and replace \FRB_1 with \frb{} throughout the simulations.

				\item
					\textbf{Extragradient method (\EG)}

					Proposed by Korpelevich \cite{korpelevich1976extragradient}, it is a classical first-order method for monotone variational inequalities with \(F\) globally Lipschitz continuous.
					Given the current iterate \(\xk\), it performs the updates
					\begin{align*}
						\bar{x}^k &= \prox_{\gamma g}\bigl(\xk - \gamma F(\xk)\bigr),\\
						\xk* &= \prox_{\gamma g}\bigl(\xk - \gamma F(\bar{x}^k)\bigr).
					\end{align*}
					The method requires two \(F\) evaluations and two proximal evaluations per iteration, with \(0<\gamma<\nicefrac{1}{L_F}\).

				\item
					\textbf{Forward-backward-forward method (\FBF)}

					Proposed by Tseng \cite{tseng2000modified}, it is another classical splitting algorithm for monotone inclusions of the form \(0\in F(x)+\partial g(x)\).
					Given the current iterate \(\xk\), the method performs the updates
					\begin{align*}
						\bar{x}^k &= \prox_{\gamma g}\bigl(\xk - \gamma F(\xk)\bigr),\\
						\xk* &= \bar{x}^k - \gamma \bigl(F(\bar{x}^k) - F(\xk)\bigr),
					\end{align*}
					and requires \(0<\gamma<\nicefrac{1}{L_F}\).
					Compared to the extragradient method, \FBF{} requires only one proximal evaluation per iteration, while still using two evaluations of \(F\).

				\item
					\textbf{\AGRAAL}

					We run \AGRAAL{} as in \eqref{eq:aGRAAL}, with \(\bar x^0=x^0\) and stepsize
					\[
						\gamk*
					=
						\min\set*{
							\omega\gamk
						,\,
							\tfrac{\varphi^2}{4\gamk'\Lk^2}
						,\,
							\gamma_{\rm max}
						},
					\]
					with \(\Lk\) as in \eqref{eq:Lk}.
					We take \(\varphi=\frac{3+\sqrt5}{4}\approx1.309\), the largest admissible \(\omega=\frac{1}{\varphi}+\frac{1}{\varphi^2}\approx1.347\), \(\gamma_0=\gamma_1=1\) and \(\gamma_{\rm max}=10^{10}\).
					Each iteration requires one evaluation of \(F\) and one proximal computation.
				\end{itemize}%

	\phantomsection
	\addcontentsline{toc}{section}{References}
	\bibliographystyle{plain}
	\bibliography{references.bib}

\end{document}